\documentclass{amsart}
\usepackage{amsfonts}
\usepackage{amsmath}
\usepackage{amssymb}
\usepackage{amsthm}
\usepackage{graphicx}
\usepackage{capt-of}
\usepackage{tikz}
\usetikzlibrary{patterns}
\usetikzlibrary{decorations.markings}
\usepackage{multirow}
\usepackage{array} 
\usetikzlibrary{calc}

\newtheorem{theorem}{Theorem}
\newtheorem{prop}{Proposition}
\newtheorem{lemma}{Lemma}

\newtheorem{exmp}{Example}

\begin{document}
\title{Markov chains of $Z$-oriented triangulations of surfaces}
\author{Adam Tyc}
%Markov chains (discrete-time Markov processes on discrete state spaces)

\
\address{Adam Tyc: Faculty of Mathematics and Computer Science, 
University of Warmia and Mazury, S{\l}oneczna 54, 10-710 Olsztyn, Poland}
\email{adam.tyc@matman.uwm.edu.pl}

\maketitle
\begin{abstract}
	We consider triangulations of closed $2$-dimensional (not necessarily orientable) surfaces. 
	Any minimal set of zigzags that double covers the set of edges provides a $z$-orientation of the triangulation. 
	We introduce Markov chains of $z$-oriented triangulations. 
	Our main result is a characterization of their ergodicity. 
	This topic is closely connected to coloring of Eulerian triangulations. 
	%Using known results on coloring of Eulerian triangulations we further simplify this criterion for the sphere and the real projective plane. 
\end{abstract}

%%%%%%%%%%%%%%%%%%%%%%%%%%%%%%%%%%%%%%%%%%%%%%%%%%%%%%%%%%%%%%%%%%%%%%%%%%%%%%%%%%%%%%%%%%%%%%%
\section{Introduction}
A {\it zigzag} (or a {\it closed left-right path} \cite{GR-book,Shank}) is a cyclic sequence of edges in a $2$-cell embedding of a connected simple finite graph in a closed $2$-dimensional surface such that each pair of consecutive edges is contained in a face, has a common vertex and no three consecutive edges of this sequence belong to the same face. 
This concept generalizes classical {\it Petrie polygons} in regular polytopes \cite{Coxeter}. 
Zigzags are closely related to the so-called {\it Gauss codes}, i.e. finite sequences of symbols where each symbol occurs precisely twice. 
Some Gauss codes can be obtained from a projection of a closed curve onto a surface, where consecutive symbols represent self-intersection points ordered along the curve.
Embeddings with a single zigzag are geometric realizations of such codes \cite{CrRos,GR-book,Lins2}. 
There are applications of zigzags in computer graphics \cite{TriApp}. 

Zigzags in $3$-regular plane graphs are related to mathematical chemistry. 
They were used to enumerate all combinatorial possibilities for fullerenes \cite{BD, DDS-book}. 
Triangulations of the sphere are dual to such objects and have the same zigzags. 
Interesting examples of spherical triangulations are {\it tetrahedral chains}, i.e. sequences of tetrahedra glued face to face in consecutive pairs. 
In \cite{T3} Markov chains are exploited to approximate the probability that randomly constructed tetrahedral chain contains precisely $k\in\{1,2,3\}$ zigzags (up to reversing). 
More results on zigzags in triangulations of arbitrary surfaces (not necessarily orientable) can be found in \cite{PT1,PT3,PT2,T2,T1}. 
Zigzags are also investigated in arbitrary embedded graphs \cite{T4} and in thin chamber complexes \cite{DezaPankov}. 

A {\it $z$-orientation} of a triangulation is a minimal set of zigzags which double covers all edges \cite{DDS-book,T1}. 
There are two possibilities for an edge in a triangulation with a fixed $z$-orientation: 
zigzags pass through this edge in different directions or twice in the same direction. 
The edge is {\it of type I} or {\it of type II} in these cases, respectively. 
Also, for faces there are two possibilities: a face is {\it of type I} if two of its edges are of type I and the other is of type II or {\it of type II} if all its edges are of type II. 
It is easy to see that every vertex of a $z$-oriented triangulation is incident to an even number of edges of type II and, consequently, the subgraph formed by all edges of type II is Eulerian. 
In particular, $z$-oriented triangulations with all faces of type II are Eulerian. 
Eulerian triangulations of surfaces are interesting, for instance, due to their properties related to graph coloring \cite{MoharKlein,MoharProj,TsaiWest}. 

In the present paper, we introduce a Markov chain for $z$-oriented triangulations of arbitrary surfaces. 
The transition graph of this chain is constructed from the triangulation graph where each edge of type II is considered as a directed edge and each edge of type I is replaced with two oppositely directed edges. 
If all faces of the triangulation are of type II, then the transition graph is identified with a {\it directed Eulerian triangulation}. 

Our main result is the following: 
the Markov chain is ergodic if and only if the $z$-oriented triangulation is not $3$-colorable or it contains a face of type I. 

Using results of Tsai and West \cite{TsaiWest}, Fisk \cite{Fisk} and Mohar \cite{MoharProj} we simplify the above characterization for the case of sphere and the real projective plane as follows: the Markov chain is ergodic if and only if its $z$-oriented triangulation contains a face of type I.

%%%%%%%%%%%%%%%%%%%%%%%%%%%%%%%%%%%%%%%%%%%%%%%%%%%%%%%%%%%%%%%%%%%%%%%%%%%%%%%%%%%%%%%%%%%%%%%
\section{Zigzags and $z$-oriented triangulations of surfaces}
Let $M$ be a connected closed $2$-dimensional surface (not necessarily orientable). 
An embedding of a graph $G$ in $M$ is {\it $2$-cell} if each connected component of $M\setminus G$ is homeomorphic to an open disk. 
The closures of these connected components are called {\it faces}. 
A $2$-cell embedding of $G$ is {\it closed} if each face is homeomorphic to a closed disk. 
The case of a digraph is the same.  

A closed $2$-cell embedding $\Gamma$ of a connected finite graph in $M$ is called a {\it multi-triangulation} of $M$ if all its faces are triangles. 
In the case when the graph of a multi-triangulation is simple, we call it a {\it triangulation}. 
For triangulations of surfaces the following hold:
\begin{enumerate}
	\item[$\bullet$] every edge is contained in precisely two distinct faces,
	\item[$\bullet$] the intersection of two distinct faces is either an edge, a vertex or empty. 
\end{enumerate} 
For multi-triangulations which are not triangulations, the second fails. 
We say that two edges are {\it adjacent} if they are distinct and there is a face containing them. 
Thus, adjacent edges have a common vertex. 
Two faces are {\it adjacent} if they are distinct and their intersection is an edge. 
 
A {\it zigzag} in a triangulation $\Gamma$ is a sequence of edges $Z=\{e_i\}_{i\in\mathbb{N}}$ satisfying the following conditions for every $i\in\mathbb{N}$:
\begin{enumerate}
	\item[$\bullet$] $e_i$ and $e_{i+1}$ are adjacent, 
	\item[$\bullet$] the faces containing $e_i, e_{i+1}$ and $e_{i+1}, e_{i+2}$ are adjacent and the edges $e_i, e_{i+2}$ are disjoint. 
\end{enumerate} 
Since $\Gamma$ is a finite graph, there is a natural number $n>0$ such that $e_{i+n}=e_i$ for every $i\in\mathbb{N}$. 
In other words, $Z$ is a cyclic sequence $\{e_1,\dots,e_n\}$, where $n$ is the smallest number satisfying this condition. 
Observe that $Z$ can be (equivalently) presented as a cyclic sequence of vertices $\{v_1,\dots,v_n\}$, where $e_i$ is formed by the vertices $v_i, v_{i+1}$ for $i=1,\dots,n-1$ and $e_n$ is formed by $v_n, v_1$. 

Every zigzag is completely determined by a pair of consecutive edges belonging to it, and conversely, any ordered pair of adjacent edges determines a unique zigzag that contains this pair. 
If $Z=\{e_1,\dots,e_n\}$ is a zigzag, then the sequence $Z^{-1}=\{e_n,\dots,e_1\}$ is also a zigzag and we call it {\it reversed} to $Z$. 
A zigzag cannot be self-reversed, since it does not contain a subsequence of the form $e,e',\dots,e',e$ (see \cite{PT2}). 
Let $\mathcal{Z}(\Gamma)$ be the set of all zigzags in $\Gamma$. 
It follows that $\mathcal{Z}(\Gamma)$ consists of an even number of zigzags. 

A {\it $z$-orientation} $\tau$ of $\Gamma$ is a set of zigzags such that for every $Z\in\mathcal{Z}(\Gamma)$ we have either $Z\in\tau$ or $Z^{-1}\in\tau$. 
If $\Gamma$ contains precisely $k$ zigzags up to reversing, i.e. $|\mathcal{Z}(\Gamma)|=2k$, then $|\tau|=k$ and there are precisely $2^k$ distinct $z$-orientations of $\Gamma$. 
If $\tau=\{Z_1,\dots,Z_k\}$ is a $z$-orientation, then we say that the $z$-orientation $\tau^{-1}=\{Z^{-1}_1,\dots,Z^{-1}_k\}$ is {\it reversed} to $\tau$. 
The triangulation $\Gamma$ with a fixed $z$-orientation $\tau$ is said to be a {\it $z$-oriented triangulation} and denoted by $(\Gamma,\tau)$ (see \cite{T2, T1}). 

Consider a $z$-oriented triangulation $(\Gamma,\tau)$. 
For each edge $e$ of $\Gamma$, one of the following possibilities is realized: 
\begin{enumerate}
	\item[$(1)$] there exists a zigzag $Z\in\tau$ such that $e$ occurs twice in $Z$ and $e$ does not occur in any other zigzag from $\tau$; 
	\item[$(2)$] there exist distinct zigzags $Z,Z'\in\tau$ such that $e$ occurs once in each of them and $e$ does not occur in any other zigzag from $\tau$. 
\end{enumerate}
In the case $(1)$, $Z$ passes through $e$ twice either in opposite directions or in the same direction. Then, we say that $e$ is of {\it type {\rm I}} or of {\it type {\rm II}}, respectively. 
Similarly, in the case $(2)$, $e$ is said to be of {\it type {\rm I}} if $Z$ and $Z'$ pass through it in opposite directions, and of {\it type {\rm II}} if $Z$ and $Z'$ pass through it in the same direction. 
Thus, edges of type II can be considered as directed edges. 
If we replace $\tau$ with $\tau^{-1}$, then the types of edges do not change, but the directions of edges of type II are reversed. 

By \cite[Proposition 1]{T1}, each face of $(\Gamma,\tau)$ satisfies precisely one of the following: 
\begin{enumerate}
	\item[(I)] two edges of the face are of type I and the third edge is of type II (see Fig. 1(I));
	\item[(II)] all edges of the face are of type II and form a directed cycle (see Fig. 1(II)). 
\end{enumerate}
\begin{center} 
	\begin{tikzpicture}[scale=0.63]
		
		\draw[fill=black] (0,2) circle (3pt);
		\draw[fill=black] (-1.7320508076,-1) circle (3pt);
		\draw[fill=black] (1.7320508076,-1) circle (3pt);
		
		\draw [thick, line width=1pt, decoration={markings,
			mark=at position 0.59 with {\arrow[scale=1.25,>=stealth]{>>}}},
			postaction={decorate}] (5.1961524228,2) -- (3.4641016152,-1);
		
		\draw [thick, line width=1pt, decoration={markings,
			mark=at position 0.59 with {\arrow[scale=1.25,>=stealth]{>>}}},
			postaction={decorate}] (3.4641016152,-1) -- (6.9282032304,-1);
		
		\draw [thick, line width=1pt, decoration={markings,
			mark=at position 0.59 with {\arrow[scale=1.25,>=stealth]{<<}}},
			postaction={decorate}] (5.1961524228,2) -- (6.9282032304,-1);
		
		\node at (0,-1.7) {(I)};
		
		\draw[fill=black] (5.1961524228,2) circle (3pt);
		\draw[fill=black] (3.4641016152,-1) circle (3pt);
		\draw[fill=black] (6.9282032304,-1) circle (3pt);
		
		\draw [thick, line width=1pt, decoration={markings,
			mark=at position 0.59 with {\arrow[scale=1.25,>=stealth]{><}}},
			postaction={decorate}] (0,2) -- (-1.7320508076,-1);
		
		\draw [thick, line width=1pt, decoration={markings,
			mark=at position 0.59 with {\arrow[scale=1.25,>=stealth]{>>}}},
			postaction={decorate}] (-1.7320508076,-1) -- (1.7320508076,-1);
		
		\draw [thick, line width=1pt, decoration={markings,
			mark=at position 0.59 with {\arrow[scale=1.25,>=stealth]{><}}},
			postaction={decorate}] (0,2) -- (1.7320508076,-1);
		
		\node at (5.1961524228,-1.7) {(II)};
	\end{tikzpicture}
	\captionof{figure}{ }
\end{center}
A face is of {\it type {\rm I}} or of {\it type {\rm II}} if the corresponding possibility is realized. 
Replacing the $z$-orientation $\tau$ with $\tau^{-1}$ does not change the types of faces. 

All digraphs that we will consider are finite and without loops, but they may contain multiple directed edges. 
An {\it Eulerian digraph} is a connected digraph such that every vertex has equal indegree and outdegree. 
A {\it directed Eulerian embedding} is a $2$-cell embedding of an Eulerian digraph in $M$ such that the boundary of each face is a directed closed walk. 
This condition is equivalent to the requirement that incoming and outgoing edges alternate in the rotation around each vertex. 
%Odwołania do def i terminologii
%Directed embeddings of 2-regular diplanar digraphs on surfaces of low Euler genus -- Definicje
%https://www.csl.mtu.edu/cs2321/www/newLectures/24_Graph_Terminology.html -- incoming, outgoing
A directed Eulerian embedding is said to be a {\it directed Eulerian multi-triangulation} 
if its underlying graph is a multi-triangulation. 
If this graph is simple (it is a triangulation), then the embedding is called a {\it directed Eulerian triangulation}. 

Let $\Gamma_{II}$ be the directed graph formed by all edges of type II from $(\Gamma,\tau)$ (their directions are determined by $\tau$) and all their vertices. 
For every vertex of $\Gamma$, the number of times that the zigzags from $\tau$ enter this vertex is equal to the number of times that they leave it. 
Therefore, the indegree equals the outdegree for every vertex of $\Gamma_{II}$ and if $\Gamma_{II}$ is connected, then it is a simple Eulerian digraph. 
In particular, if all faces of $(\Gamma,\tau)$ are of type II, then $\Gamma_{II}$ is a directed Eulerian triangulation of $M$ whose underlying graph is $\Gamma$.  

\begin{exmp}\label{ex1}{\rm 
	Denote by $\mathcal{O}$ the graph of the octahedron shown in Fig. 2.
	\begin{center} 
		\begin{tikzpicture}[scale=0.42] 
			%\draw[step=0.5, gray, thin] (-15,0) grid (15,8);
			%\draw[step=1, red, thin] (-15,0) grid (15,8);
			%\draw[step=5, blue, thin] (-15,0) grid (15,8);	
			%1
			\begin{scope}[xshift=-10cm]
				\coordinate (A1) at (90:5cm);
				\coordinate (A2) at (210:5cm);
				\coordinate (A3) at (330:5cm);
				
				\draw[fill=black] (A1) circle (3.5pt);
				\draw[fill=black] (A2) circle (3.5pt);	
				\draw[fill=black] (A3) circle (3.5pt);
				
				\coordinate (B1) at (270:1.25cm);
				\coordinate (B2) at (30:1.25cm);
				\coordinate (B3) at (150:1.25cm);
				
				\draw[fill=black] (B1) circle (3.5pt);
				\draw[fill=black] (B2) circle (3.5pt);	
				\draw[fill=black] (B3) circle (3.5pt);
				
				\draw[thick, line width=1pt, decoration={markings,
					mark=at position 0.55 with {\arrow[scale=1,>=stealth]{><}}},
				postaction={decorate}] (A1) -- (A2);
				\draw[thick, line width=1pt, decoration={markings,
					mark=at position 0.55 with {\arrow[scale=1,>=stealth]{>>}}},
				postaction={decorate}] (A2) -- (A3);
				\draw[thick, line width=1pt, decoration={markings,
					mark=at position 0.55 with {\arrow[scale=1,>=stealth]{><}}},
				postaction={decorate}] (A3) -- (A1);
				
				\draw[thick, line width=1pt, decoration={markings,
					mark=at position 0.65 with {\arrow[scale=1,>=stealth]{><}}},
				postaction={decorate}] (B1) -- (B2);
				\draw[thick, line width=1pt, decoration={markings,
					mark=at position 0.67 with {\arrow[scale=1,>=stealth]{>>}}},
				postaction={decorate}] (B2) -- (B3);
				\draw[thick, line width=1pt, decoration={markings,
					mark=at position 0.65 with {\arrow[scale=1,>=stealth]{><}}},
				postaction={decorate}] (B3) -- (B1);
				
				\draw[thick, line width=1pt, decoration={markings,
					mark=at position 0.65 with {\arrow[scale=1,>=stealth]{><}}},
				postaction={decorate}] (A1) -- (B2);
				\draw[thick, line width=1pt, decoration={markings,
					mark=at position 0.65 with {\arrow[scale=1,>=stealth]{><}}},
				postaction={decorate}] (A1) -- (B3);
				
				\draw[thick, line width=1pt, decoration={markings,
					mark=at position 0.65 with {\arrow[scale=1,>=stealth]{><}}},
				postaction={decorate}] (A2) -- (B1);
				\draw[thick, line width=1pt, decoration={markings,
					mark=at position 0.65 with {\arrow[scale=1,>=stealth]{<<}}},
				postaction={decorate}] (A2) -- (B3);
				
				\draw[thick, line width=1pt, decoration={markings,
					mark=at position 0.65 with {\arrow[scale=1,>=stealth]{><}}},
				postaction={decorate}] (A3) -- (B1);
				\draw[thick, line width=1pt, decoration={markings,
					mark=at position 0.65 with {\arrow[scale=1,>=stealth]{>>}}},
				postaction={decorate}] (A3) -- (B2);
				
				\node at (0,-3.7cm) {$(\mathcal{O},\tau_1)$};
				
				\node at (90:5.5cm) {$a_1$};
				\node at (215:5.5cm) {$a_2$};
				\node at (325:5.5cm) {$a_3$};
				
				\node at (30:1.8cm) {$b_2$};
				\node at (150:1.8cm) {$b_3$};
				\node at (270:1.8cm) {$b_1$};		
			\end{scope}	
			%2
			\begin{scope}
				\coordinate (A1) at (90:5cm);
				\coordinate (A2) at (210:5cm);
				\coordinate (A3) at (330:5cm);
				
				\draw[fill=black] (A1) circle (3.5pt);
				\draw[fill=black] (A2) circle (3.5pt);	
				\draw[fill=black] (A3) circle (3.5pt);
				
				\coordinate (B1) at (270:1.25cm);
				\coordinate (B2) at (30:1.25cm);
				\coordinate (B3) at (150:1.25cm);
				
				\draw[fill=black] (B1) circle (3.5pt);
				\draw[fill=black] (B2) circle (3.5pt);	
				\draw[fill=black] (B3) circle (3.5pt);
				
				\draw[thick, line width=1pt, decoration={markings,
					mark=at position 0.55 with {\arrow[scale=1,>=stealth]{>>}}},
				postaction={decorate}] (A1) -- (A2);
				\draw[thick, line width=1pt, decoration={markings,
					mark=at position 0.55 with {\arrow[scale=1,>=stealth]{>>}}},
				postaction={decorate}] (A2) -- (A3);
				\draw[thick, line width=1pt, decoration={markings,
					mark=at position 0.55 with {\arrow[scale=1,>=stealth]{>>}}},
				postaction={decorate}] (A3) -- (A1);
				
				\draw[thick, line width=1pt, decoration={markings,
					mark=at position 0.67 with {\arrow[scale=1,>=stealth]{>>}}},
				postaction={decorate}] (B1) -- (B2);
				\draw[thick, line width=1pt, decoration={markings,
					mark=at position 0.67 with {\arrow[scale=1,>=stealth]{>>}}},
				postaction={decorate}] (B2) -- (B3);
				\draw[thick, line width=1pt, decoration={markings,
					mark=at position 0.67 with {\arrow[scale=1,>=stealth]{>>}}},
				postaction={decorate}] (B3) -- (B1);
				
				\draw[thick, line width=1pt, decoration={markings,
					mark=at position 0.65 with {\arrow[scale=1,>=stealth]{>>}}},
				postaction={decorate}] (A1) -- (B2);
				\draw[thick, line width=1pt, decoration={markings,
					mark=at position 0.65 with {\arrow[scale=1,>=stealth]{<<}}},
				postaction={decorate}] (A1) -- (B3);
				
				\draw[thick, line width=1pt, decoration={markings,
					mark=at position 0.65 with {\arrow[scale=1,>=stealth]{<<}}},
				postaction={decorate}] (A2) -- (B1);
				\draw[thick, line width=1pt, decoration={markings,
					mark=at position 0.65 with {\arrow[scale=1,>=stealth]{>>}}},
				postaction={decorate}] (A2) -- (B3);
				
				\draw[thick, line width=1pt, decoration={markings,
					mark=at position 0.65 with {\arrow[scale=1,>=stealth]{>>}}},
				postaction={decorate}] (A3) -- (B1);
				\draw[thick, line width=1pt, decoration={markings,
					mark=at position 0.65 with {\arrow[scale=1,>=stealth]{<<}}},
				postaction={decorate}] (A3) -- (B2);
				
				\node at (0,-3.7cm) {$(\mathcal{O},\tau_2)$};
				
				\node at (90:5.5cm) {$a_1$};
				\node at (215:5.5cm) {$a_2$};
				\node at (325:5.5cm) {$a_3$};
				
				\node at (30:1.8cm) {$b_2$};
				\node at (150:1.8cm) {$b_3$};
				\node at (270:1.8cm) {$b_1$};	
			\end{scope}	
			%3
			\begin{scope}[xshift=10cm]
				\coordinate (A1) at (90:5cm);
				\coordinate (A2) at (210:5cm);
				\coordinate (A3) at (330:5cm);
				
				\draw[fill=black] (A1) circle (3.5pt);
				\draw[fill=black] (A2) circle (3.5pt);	
				\draw[fill=black] (A3) circle (3.5pt);
				
				\coordinate (B1) at (270:1.25cm);
				\coordinate (B2) at (30:1.25cm);
				\coordinate (B3) at (150:1.25cm);
				
				\draw[fill=black] (B1) circle (3.5pt);
				\draw[fill=black] (B2) circle (3.5pt);	
				\draw[fill=black] (B3) circle (3.5pt);
				
				\draw[thick, line width=1pt, decoration={markings,
					mark=at position 0.55 with {\arrow[scale=1,>=stealth]{><}}},
				postaction={decorate}] (A1) -- (A2);
				\draw[thick, line width=1pt, decoration={markings,
					mark=at position 0.55 with {\arrow[scale=1,>=stealth]{>>}}},
				postaction={decorate}] (A2) -- (A3);
				\draw[thick, line width=1pt, decoration={markings,
					mark=at position 0.55 with {\arrow[scale=1,>=stealth]{><}}},
				postaction={decorate}] (A3) -- (A1);
				
				\draw[thick, line width=1pt, decoration={markings,
					mark=at position 0.65 with {\arrow[scale=1,>=stealth]{><}}},
				postaction={decorate}] (B1) -- (B2);
				\draw[thick, line width=1pt, decoration={markings,
					mark=at position 0.67 with {\arrow[scale=1,>=stealth]{>>}}},
				postaction={decorate}] (B2) -- (B3);
				\draw[thick, line width=1pt, decoration={markings,
					mark=at position 0.65 with {\arrow[scale=1,>=stealth]{><}}},
				postaction={decorate}] (B3) -- (B1);
				
				\draw[thick, line width=1pt, decoration={markings,
					mark=at position 0.65 with {\arrow[scale=1,>=stealth]{>>}}},
				postaction={decorate}] (A1) -- (B2);
				\draw[thick, line width=1pt, decoration={markings,
					mark=at position 0.65 with {\arrow[scale=1,>=stealth]{<<}}},
				postaction={decorate}] (A1) -- (B3);
				
				\draw[thick, line width=1pt, decoration={markings,
					mark=at position 0.65 with {\arrow[scale=1,>=stealth]{<<}}},
				postaction={decorate}] (A2) -- (B1);
				\draw[thick, line width=1pt, decoration={markings,
					mark=at position 0.65 with {\arrow[scale=1,>=stealth]{><}}},
				postaction={decorate}] (A2) -- (B3);
				
				\draw[thick, line width=1pt, decoration={markings,
					mark=at position 0.65 with {\arrow[scale=1,>=stealth]{>>}}},
				postaction={decorate}] (A3) -- (B1);
				\draw[thick, line width=1pt, decoration={markings,
					mark=at position 0.65 with {\arrow[scale=1,>=stealth]{><}}},
				postaction={decorate}] (A3) -- (B2);
				
				\node at (0,-3.7cm) {$(\mathcal{O},\tau_3)$};
				
				\node at (90:5.5cm) {$a_1$};
				\node at (215:5.5cm) {$a_2$};
				\node at (325:5.5cm) {$a_3$};
				
				\node at (30:1.8cm) {$b_2$};
				\node at (150:1.8cm) {$b_3$};
				\node at (270:1.8cm) {$b_1$};
			\end{scope}	
		\end{tikzpicture}
		\captionof{figure}{ }
	\end{center}
	This graph contains precisely $4$ zigzags (up to reversing):
	$$Z_1=a_1,a_2,a_3,b_1,b_2,b_3;\quad Z_2=a_2,a_3,a_1,b_2,b_3,b_1;$$ $$Z_3=a_3,a_1,a_2,b_3,b_1,b_2;\quad Z_4=a_1,b_2,a_3,b_1,a_2,b_3.$$
	Consider the following $3$ $z$-orientations of $\mathcal{O}$:
	$$\tau_1=\{Z_1,Z_2,Z^{-1}_3,Z^{-1}_4\},\quad 
	\tau_2=\{Z_1,Z_2,Z_3,Z_4\},\quad 
	\tau_3=\{Z_1,Z_2,Z^{-1}_3,Z_4\}.$$ 
	For $\tau_1$ all faces of $\mathcal{O}$ are of type I and for $\tau_2$ all faces are of type II. 
	If we choose $\tau_3$, then there are faces of both types (see Fig. 2). 
}\end{exmp}

\begin{exmp}\label{ex2}{\rm 
	Let $\mathcal{T}_{k,m}$ (where $k,m\geq 3$) be a triangulation of the torus obtained by identifying opposite sides of a $k\times m$ rectangular grid, where each square is subdivided by adding a diagonal edge connecting its top-left and bottom-right vertices. 
	The vertex set of $\mathcal{T}_{k,m}$ is $\mathbb{Z}_k\times\mathbb{Z}_m$. 
	We will write $ij$ instead of $(i,j)$ for its elements (see Fig. 3). 
	\begin{center}
		\begin{tikzpicture}[scale=0.48, xshift=0cm]
			
			%LD
			\draw[fill=black] (0,0) circle (3pt);
			\draw[fill=black] (2,0) circle (3pt);
			\draw[fill=black] (4,0) circle (3pt);
			
			\draw[fill=black] (0,2) circle (3pt);
			\draw[fill=black] (2,2) circle (3pt);
			\draw[fill=black] (4,2) circle (3pt);
			
			\draw[fill=black] (0,4) circle (3pt);
			\draw[fill=black] (2,4) circle (3pt);
			\draw[fill=black] (4,4) circle (3pt);
			
			\draw[thick, line width=1pt] (0,0) -- (4,0) -- (4,4) -- (0,4) -- cycle;
			\draw[thick, line width=1pt] (0,2) -- (4,2);
			\draw[thick, line width=1pt] (2,0) -- (2,4);
			
			\draw[thick, line width=1pt] (0,4) -- (4,0);
			\draw[thick, line width=1pt] (0,2) -- (2,0);
			\draw[thick, line width=1pt] (2,4) -- (4,2);
						
			%LU
			\draw[fill=black] (0,8) circle (3pt);
			\draw[fill=black] (2,8) circle (3pt);
			\draw[fill=black] (4,8) circle (3pt);
			
			\draw[fill=black] (0,10) circle (3pt);
			\draw[fill=black] (2,10) circle (3pt);
			\draw[fill=black] (4,10) circle (3pt);
			
			\draw[thick, line width=1pt] (0,8) -- (4,8) -- (4,10) -- (0,10) -- cycle;
			\draw[thick, line width=1pt] (0,10) -- (4,10);
			\draw[thick, line width=1pt] (2,8) -- (2,10);
			
			\draw[thick, line width=1pt] (2,10) -- (4,8);
			\draw[thick, line width=1pt] (0,10) -- (2,8);
			
			%RD
			\draw[fill=black] (8,0) circle (3pt);
			\draw[fill=black] (10,0) circle (3pt);
			
			\draw[fill=black] (8,2) circle (3pt);
			\draw[fill=black] (10,2) circle (3pt);
			
			\draw[fill=black] (8,4) circle (3pt);
			\draw[fill=black] (10,4) circle (3pt);
			
			\draw[thick, line width=1pt] (8,0) -- (10,0) -- (10,4) -- (8,4) -- cycle;
			\draw[thick, line width=1pt] (8,2) -- (10,2);
			
			\draw[thick, line width=1pt] (8,4) -- (10,2);
			\draw[thick, line width=1pt] (8,2) -- (10,0);
			
			%RU
			\draw[fill=black] (8,8) circle (3pt);
			\draw[fill=black] (10,8) circle (3pt);
			
			\draw[fill=black] (8,10) circle (3pt);
			\draw[fill=black] (10,10) circle (3pt);
			
			\draw[thick, line width=1pt] (8,8) -- (10,8) -- (10,10) -- (8,10) -- cycle;
			\draw[thick, line width=1pt] (8,10) -- (10,8);
			
			%%arrows
			\draw [->, line width=1pt]  (-0.3,0) -- (-0.3, 10);
			\draw [->, line width=1pt]  (10.3,0) -- (10.3, 10);
			
			\draw [->, line width=1pt]  (0,-0.3) -- (10, -0.3);
			\draw [->, line width=1pt]  (0, 10.3) -- (10, 10.3);

			\node at (0,-0.73) {$0$};
			\node at (2,-0.73) {$1$};
			\node at (4,-0.73) {$2$};
			\node at (6,-0.73) {$i$};
			\node at (8,-0.73) {$k-1$};
			\node at (10,-0.73) {$0$};
			
			\node at (-0.65,0) {$0$};
			\node at (-0.65,2) {$1$};
			\node at (-0.65,4) {$2$};
			\node at (-0.65,6) {$j$};
			\node at (-1.3,8) {$m-1$};
			\node at (-0.65,10) {$0$};

			\node[white] at (12.2,8) {.}; %
			
			%|||
			\draw[thick, line width=1pt, dashed] (0,4) -- (0,8);
			
			\draw[thick, line width=1pt, dashed] (2,4) -- (2,8);
			
			\draw[thick, line width=1pt, dashed] (4,4) -- (4,8);
			
			\draw[thick, line width=1pt, dashed] (8,4) -- (8,8);
			
			\draw[thick, line width=1pt, dashed] (10,4) -- (10,8);
			
			%---
			\draw[thick, line width=1pt, dashed] (4,0) -- (8,0);
			
			\draw[thick, line width=1pt, dashed] (4,2) -- (8,2);
			
			\draw[thick, line width=1pt, dashed] (4,4) -- (8,4);
			
			\draw[thick, line width=1pt, dashed] (4,8) -- (8,8);
			
			\draw[thick, line width=1pt, dashed] (4,10) -- (8,10);
			
			%\\\
			\draw[thick, line width=1pt, dashed] (1.6,4.4) -- (2,4);
			\draw[thick, line width=1pt, dashed] (3.6,4.4) -- (4,4);
			\draw[thick, line width=1pt, dashed] (7.6,4.4) -- (8,4);
			\draw[thick, line width=1pt, dashed] (9.6,4.4) -- (10,4);
			
			\draw[thick, line width=1pt, dashed] (0,8) -- (0.4,7.6);
			\draw[thick, line width=1pt, dashed] (2,8) -- (2.4,7.6);
			\draw[thick, line width=1pt, dashed] (4,8) -- (4.4,7.6);
			\draw[thick, line width=1pt, dashed] (8,8) -- (8.4,7.6);
			
			\draw[thick, line width=1pt, dashed] (4,10) -- (4.4,9.6);
			\draw[thick, line width=1pt, dashed] (4,4) -- (4.4,3.6);
			\draw[thick, line width=1pt, dashed] (4,2) -- (4.4,1.6);						
			
			\draw[thick, line width=1pt, dashed] (7.6,8.4) -- (8,8);
			\draw[thick, line width=1pt, dashed] (7.6,2.4) -- (8,2);
			\draw[thick, line width=1pt, dashed] (7.6,0.4) -- (8,0);
			
			%i,j
			\draw[thick, line width=1pt, dashed] (6,0) -- (6,10);
			\draw[thick, line width=1pt, dashed] (0,6) -- (10,6);
			
			\draw[thick, line width=1pt] (6,5.2) -- (6,6.8);
			\draw[thick, line width=1pt] (5.2,6) -- (6.8,6);
			
			\draw[fill=black] (6,6) circle (3pt);
			
			\draw[thick, line width=1pt, dashed] (6,6) -- (6.4,5.6);
			\draw[thick, line width=1pt, dashed] (5.6,6.4) -- (6,6);
			
			\node at (6.38,6.38) {$ij$};
			
		\end{tikzpicture}
		\captionof{figure}{ }
	\end{center}
	There are the following zigzags in $\mathcal{T}_{k,m}$ (as sequences of vertices):
	\begin{enumerate}
	\item[$\bullet$] the vertical zigzags 
	$$Z'_i=i0,(i+1)0,i1,(i+1)1,\dots,i(m-1),(i+1)(m-1)$$
	for every $i\in\mathbb{Z}_k$,
	\item[$\bullet$] the horizontal zigzags $$Z''_j=0j,0(j+1),1j,1(j+1),\dots,(k-1)j,(k-1)(j+1)$$
	for every $j\in\mathbb{Z}_m$,
	\item[$\bullet$] the skew zigzags $$Z'''_s=0(s+1),0s,1s,1(s-1),\dots,(k-1)(s+2),(k-1)(s+1)$$
	for every $s$ satisfying $0\leq s\leq \gcd(k,m)$
	\end{enumerate}
	and their reverses. 
	\begin{center}
		\begin{tikzpicture}[scale=0.53]
		%1
		\begin{scope}[xshift=-8cm]
			\draw[fill=black] (0,0) circle (3pt);
			\draw[fill=black] (2,0) circle (3pt);
			\draw[fill=black] (4,0) circle (3pt);
			\draw[fill=black] (6,0) circle (3pt);
			
			\draw[fill=black] (0,2) circle (3pt);
			\draw[fill=black] (2,2) circle (3pt);
			\draw[fill=black] (4,2) circle (3pt);
			\draw[fill=black] (6,2) circle (3pt);
			
			\draw[fill=black] (0,4) circle (3pt);
			\draw[fill=black] (2,4) circle (3pt);
			\draw[fill=black] (4,4) circle (3pt);
			\draw[fill=black] (6,4) circle (3pt);
			
			\draw[fill=black] (0,6) circle (3pt);
			\draw[fill=black] (2,6) circle (3pt);
			\draw[fill=black] (4,6) circle (3pt);
			\draw[fill=black] (6,6) circle (3pt);
			
			%poziome
			\draw[thick, line width=1pt, decoration={markings,
				mark=at position 0.65 with {\arrow[scale=1,>=stealth]{>>}}},
			postaction={decorate}] (0,0) -- (2,0);
			\draw[thick, line width=1pt, decoration={markings,
				mark=at position 0.65 with {\arrow[scale=1,>=stealth]{>>}}},
			postaction={decorate}] (2,0) -- (4,0);
			\draw[thick, line width=1pt, decoration={markings,
				mark=at position 0.65 with {\arrow[scale=1,>=stealth]{>>}}},
			postaction={decorate}] (4,0) -- (6,0);
			
			\draw[thick, line width=1pt, decoration={markings,
				mark=at position 0.65 with {\arrow[scale=1,>=stealth]{>>}}},
			postaction={decorate}] (0,2) -- (2,2);
			\draw[thick, line width=1pt, decoration={markings,
				mark=at position 0.65 with {\arrow[scale=1,>=stealth]{>>}}},
			postaction={decorate}] (2,2) -- (4,2);
			\draw[thick, line width=1pt, decoration={markings,
				mark=at position 0.65 with {\arrow[scale=1,>=stealth]{>>}}},
			postaction={decorate}] (4,2) -- (6,2);
			
			\draw[thick, line width=1pt, decoration={markings,
				mark=at position 0.65 with {\arrow[scale=1,>=stealth]{>>}}},
			postaction={decorate}] (0,4) -- (2,4);
			\draw[thick, line width=1pt, decoration={markings,
				mark=at position 0.65 with {\arrow[scale=1,>=stealth]{>>}}},
			postaction={decorate}] (2,4) -- (4,4);
			\draw[thick, line width=1pt, decoration={markings,
				mark=at position 0.65 with {\arrow[scale=1,>=stealth]{>>}}},
			postaction={decorate}] (4,4) -- (6,4);
			
			\draw[thick, line width=1pt, decoration={markings,
				mark=at position 0.65 with {\arrow[scale=1,>=stealth]{>>}}},
			postaction={decorate}] (0,6) -- (2,6);
			\draw[thick, line width=1pt, decoration={markings,
				mark=at position 0.65 with {\arrow[scale=1,>=stealth]{>>}}},
			postaction={decorate}] (2,6) -- (4,6);
			\draw[thick, line width=1pt, decoration={markings,
				mark=at position 0.65 with {\arrow[scale=1,>=stealth]{>>}}},
			postaction={decorate}] (4,6) -- (6,6);
			
			%pionowe
			\draw[thick, line width=1pt, decoration={markings,
				mark=at position 0.63 with {\arrow[scale=1,>=stealth]{><}}},
			postaction={decorate}] (0,0) -- (0,2);
			\draw[thick, line width=1pt, decoration={markings,
				mark=at position 0.63 with {\arrow[scale=1,>=stealth]{><}}},
			postaction={decorate}] (0,2) -- (0,4);
			\draw[thick, line width=1pt, decoration={markings,
				mark=at position 0.63 with {\arrow[scale=1,>=stealth]{><}}},
			postaction={decorate}] (0,4) -- (0,6);
			
			\draw[thick, line width=1pt, decoration={markings,
				mark=at position 0.63 with {\arrow[scale=1,>=stealth]{><}}},
			postaction={decorate}] (2,0) -- (2,2);
			\draw[thick, line width=1pt, decoration={markings,
				mark=at position 0.63 with {\arrow[scale=1,>=stealth]{><}}},
			postaction={decorate}] (2,2) -- (2,4);
			\draw[thick, line width=1pt, decoration={markings,
				mark=at position 0.63 with {\arrow[scale=1,>=stealth]{><}}},
			postaction={decorate}] (2,4) -- (2,6);
			
			\draw[thick, line width=1pt, decoration={markings,
				mark=at position 0.63 with {\arrow[scale=1,>=stealth]{><}}},
			postaction={decorate}] (4,0) -- (4,2);
			\draw[thick, line width=1pt, decoration={markings,
				mark=at position 0.63 with {\arrow[scale=1,>=stealth]{><}}},
			postaction={decorate}] (4,2) -- (4,4);
			\draw[thick, line width=1pt, decoration={markings,
				mark=at position 0.63 with {\arrow[scale=1,>=stealth]{><}}},
			postaction={decorate}] (4,4) -- (4,6);
			
			\draw[thick, line width=1pt, decoration={markings,
				mark=at position 0.63 with {\arrow[scale=1,>=stealth]{><}}},
			postaction={decorate}] (6,0) -- (6,2);
			\draw[thick, line width=1pt, decoration={markings,
				mark=at position 0.63 with {\arrow[scale=1,>=stealth]{><}}},
			postaction={decorate}] (6,2) -- (6,4);
			\draw[thick, line width=1pt, decoration={markings,
				mark=at position 0.63 with {\arrow[scale=1,>=stealth]{><}}},
			postaction={decorate}] (6,4) -- (6,6);
			
			%skosne						
			\draw[thick, line width=1pt, decoration={markings,
				mark=at position 0.6 with {\arrow[scale=1,>=stealth]{><}}},
			postaction={decorate}] (0,2) -- (2,0);
			
			\draw[thick, line width=1pt, decoration={markings,
				mark=at position 0.6 with {\arrow[scale=1,>=stealth]{><}}},
			postaction={decorate}] (0,4) -- (2,2);
			\draw[thick, line width=1pt, decoration={markings,
				mark=at position 0.6 with {\arrow[scale=1,>=stealth]{><}}},
			postaction={decorate}] (2,2) -- (4,0);
			
			\draw[thick, line width=1pt, decoration={markings,
				mark=at position 0.6 with {\arrow[scale=1,>=stealth]{><}}},
			postaction={decorate}] (0,6) -- (2,4);
			\draw[thick, line width=1pt, decoration={markings,
				mark=at position 0.6 with {\arrow[scale=1,>=stealth]{><}}},
			postaction={decorate}] (2,4) -- (4,2);
			\draw[thick, line width=1pt, decoration={markings,
				mark=at position 0.6 with {\arrow[scale=1,>=stealth]{><}}},
			postaction={decorate}] (4,2) -- (6,0);
			
			\draw[thick, line width=1pt, decoration={markings,
				mark=at position 0.6 with {\arrow[scale=1,>=stealth]{><}}},
			postaction={decorate}] (2,6) -- (4,4);
			\draw[thick, line width=1pt, decoration={markings,
				mark=at position 0.6 with {\arrow[scale=1,>=stealth]{><}}},
			postaction={decorate}] (4,4) -- (6,2);
			
			\draw[thick, line width=1pt, decoration={markings,
				mark=at position 0.6 with {\arrow[scale=1,>=stealth]{><}}},
			postaction={decorate}] (4,6) -- (6,4);

			%%arrows
			\draw [->, line width=1pt]  (-0.3,0) -- (-0.3, 6);
			\draw [->, line width=1pt]  (6.3,0) -- (6.3, 6);
			
			\draw [->, line width=1pt]  (0,-0.3) -- (6, -0.3);
			\draw [->, line width=1pt]  (0, 6.3) -- (6, 6.3);

			\node at (0,-0.7) {$0$};
			\node at (2,-0.7) {$1$};
			\node at (4,-0.7) {$2$};
			\node at (6,-0.7) {$0$};
			
			\node at (-0.65,0) {$0$};
			\node at (-0.65,2) {$1$};
			\node at (-0.65,4) {$2$};
			\node at (-0.65,6) {$0$};
			
			\node at (3,-1.5) {$(\mathcal{T}_{3,3},\tau_1)$};
		\end{scope}
		%2
		\begin{scope}
			\draw[fill=black] (0,0) circle (3pt);
			\draw[fill=black] (2,0) circle (3pt);
			\draw[fill=black] (4,0) circle (3pt);
			\draw[fill=black] (6,0) circle (3pt);
			
			\draw[fill=black] (0,2) circle (3pt);
			\draw[fill=black] (2,2) circle (3pt);
			\draw[fill=black] (4,2) circle (3pt);
			\draw[fill=black] (6,2) circle (3pt);
			
			\draw[fill=black] (0,4) circle (3pt);
			\draw[fill=black] (2,4) circle (3pt);
			\draw[fill=black] (4,4) circle (3pt);
			\draw[fill=black] (6,4) circle (3pt);
			
			\draw[fill=black] (0,6) circle (3pt);
			\draw[fill=black] (2,6) circle (3pt);
			\draw[fill=black] (4,6) circle (3pt);
			\draw[fill=black] (6,6) circle (3pt);
			
			%poziome
			\draw[thick, line width=1pt, decoration={markings,
				mark=at position 0.65 with {\arrow[scale=1,>=stealth]{>>}}},
			postaction={decorate}] (0,0) -- (2,0);
			\draw[thick, line width=1pt, decoration={markings,
				mark=at position 0.65 with {\arrow[scale=1,>=stealth]{>>}}},
			postaction={decorate}] (2,0) -- (4,0);
			\draw[thick, line width=1pt, decoration={markings,
				mark=at position 0.65 with {\arrow[scale=1,>=stealth]{>>}}},
			postaction={decorate}] (4,0) -- (6,0);
			
			\draw[thick, line width=1pt, decoration={markings,
				mark=at position 0.65 with {\arrow[scale=1,>=stealth]{>>}}},
			postaction={decorate}] (0,2) -- (2,2);
			\draw[thick, line width=1pt, decoration={markings,
				mark=at position 0.65 with {\arrow[scale=1,>=stealth]{>>}}},
			postaction={decorate}] (2,2) -- (4,2);
			\draw[thick, line width=1pt, decoration={markings,
				mark=at position 0.65 with {\arrow[scale=1,>=stealth]{>>}}},
			postaction={decorate}] (4,2) -- (6,2);
			
			\draw[thick, line width=1pt, decoration={markings,
				mark=at position 0.65 with {\arrow[scale=1,>=stealth]{>>}}},
			postaction={decorate}] (0,4) -- (2,4);
			\draw[thick, line width=1pt, decoration={markings,
				mark=at position 0.65 with {\arrow[scale=1,>=stealth]{>>}}},
			postaction={decorate}] (2,4) -- (4,4);
			\draw[thick, line width=1pt, decoration={markings,
				mark=at position 0.65 with {\arrow[scale=1,>=stealth]{>>}}},
			postaction={decorate}] (4,4) -- (6,4);
			
			\draw[thick, line width=1pt, decoration={markings,
				mark=at position 0.65 with {\arrow[scale=1,>=stealth]{>>}}},
			postaction={decorate}] (0,6) -- (2,6);
			\draw[thick, line width=1pt, decoration={markings,
				mark=at position 0.65 with {\arrow[scale=1,>=stealth]{>>}}},
			postaction={decorate}] (2,6) -- (4,6);
			\draw[thick, line width=1pt, decoration={markings,
				mark=at position 0.65 with {\arrow[scale=1,>=stealth]{>>}}},
			postaction={decorate}] (4,6) -- (6,6);
			
			%pionowe
			\draw[thick, line width=1pt, decoration={markings,
				mark=at position 0.63 with {\arrow[scale=1,>=stealth]{<<}}},
			postaction={decorate}] (0,0) -- (0,2);
			\draw[thick, line width=1pt, decoration={markings,
				mark=at position 0.63 with {\arrow[scale=1,>=stealth]{<<}}},
			postaction={decorate}] (0,2) -- (0,4);
			\draw[thick, line width=1pt, decoration={markings,
				mark=at position 0.63 with {\arrow[scale=1,>=stealth]{<<}}},
			postaction={decorate}] (0,4) -- (0,6);
			
			\draw[thick, line width=1pt, decoration={markings,
				mark=at position 0.63 with {\arrow[scale=1,>=stealth]{<<}}},
			postaction={decorate}] (2,0) -- (2,2);
			\draw[thick, line width=1pt, decoration={markings,
				mark=at position 0.63 with {\arrow[scale=1,>=stealth]{<<}}},
			postaction={decorate}] (2,2) -- (2,4);
			\draw[thick, line width=1pt, decoration={markings,
				mark=at position 0.63 with {\arrow[scale=1,>=stealth]{<<}}},
			postaction={decorate}] (2,4) -- (2,6);
			
			\draw[thick, line width=1pt, decoration={markings,
				mark=at position 0.63 with {\arrow[scale=1,>=stealth]{<<}}},
			postaction={decorate}] (4,0) -- (4,2);
			\draw[thick, line width=1pt, decoration={markings,
				mark=at position 0.63 with {\arrow[scale=1,>=stealth]{<<}}},
			postaction={decorate}] (4,2) -- (4,4);
			\draw[thick, line width=1pt, decoration={markings,
				mark=at position 0.63 with {\arrow[scale=1,>=stealth]{<<}}},
			postaction={decorate}] (4,4) -- (4,6);
			
			\draw[thick, line width=1pt, decoration={markings,
				mark=at position 0.63 with {\arrow[scale=1,>=stealth]{<<}}},
			postaction={decorate}] (6,0) -- (6,2);
			\draw[thick, line width=1pt, decoration={markings,
				mark=at position 0.63 with {\arrow[scale=1,>=stealth]{<<}}},
			postaction={decorate}] (6,2) -- (6,4);
			\draw[thick, line width=1pt, decoration={markings,
				mark=at position 0.63 with {\arrow[scale=1,>=stealth]{<<}}},
			postaction={decorate}] (6,4) -- (6,6);
			
			%skosne						
			\draw[thick, line width=1pt, decoration={markings,
				mark=at position 0.6 with {\arrow[scale=1,>=stealth]{<<}}},
			postaction={decorate}] (0,2) -- (2,0);
			
			\draw[thick, line width=1pt, decoration={markings,
				mark=at position 0.6 with {\arrow[scale=1,>=stealth]{<<}}},
			postaction={decorate}] (0,4) -- (2,2);
			\draw[thick, line width=1pt, decoration={markings,
				mark=at position 0.6 with {\arrow[scale=1,>=stealth]{<<}}},
			postaction={decorate}] (2,2) -- (4,0);
			
			\draw[thick, line width=1pt, decoration={markings,
				mark=at position 0.6 with {\arrow[scale=1,>=stealth]{<<}}},
			postaction={decorate}] (0,6) -- (2,4);
			\draw[thick, line width=1pt, decoration={markings,
				mark=at position 0.6 with {\arrow[scale=1,>=stealth]{<<}}},
			postaction={decorate}] (2,4) -- (4,2);
			\draw[thick, line width=1pt, decoration={markings,
				mark=at position 0.6 with {\arrow[scale=1,>=stealth]{<<}}},
			postaction={decorate}] (4,2) -- (6,0);
			
			\draw[thick, line width=1pt, decoration={markings,
				mark=at position 0.6 with {\arrow[scale=1,>=stealth]{<<}}},
			postaction={decorate}] (2,6) -- (4,4);
			\draw[thick, line width=1pt, decoration={markings,
				mark=at position 0.6 with {\arrow[scale=1,>=stealth]{<<}}},
			postaction={decorate}] (4,4) -- (6,2);
			
			\draw[thick, line width=1pt, decoration={markings,
				mark=at position 0.6 with {\arrow[scale=1,>=stealth]{<<}}},
			postaction={decorate}] (4,6) -- (6,4);

			%%arrows
			\draw [->, line width=1pt]  (-0.3,0) -- (-0.3, 6);
			\draw [->, line width=1pt]  (6.3,0) -- (6.3, 6);
			
			\draw [->, line width=1pt]  (0,-0.3) -- (6, -0.3);
			\draw [->, line width=1pt]  (0, 6.3) -- (6, 6.3);

			\node at (0,-0.7) {$0$};
			\node at (2,-0.7) {$1$};
			\node at (4,-0.7) {$2$};
			\node at (6,-0.7) {$0$};
			
			\node at (-0.65,0) {$0$};
			\node at (-0.65,2) {$1$};
			\node at (-0.65,4) {$2$};
			\node at (-0.65,6) {$0$};
			
			\node at (3,-1.5) {$(\mathcal{T}_{3,3},\tau_2)$};
		\end{scope}
		%3
		\begin{scope}[xshift=8cm,yshift=-1cm]
			\draw[fill=black] (0,0) circle (3pt);
			\draw[fill=black] (2,0) circle (3pt);
			\draw[fill=black] (4,0) circle (3pt);
			\draw[fill=black] (6,0) circle (3pt);
			
			\draw[fill=black] (0,2) circle (3pt);
			\draw[fill=black] (2,2) circle (3pt);
			\draw[fill=black] (4,2) circle (3pt);
			\draw[fill=black] (6,2) circle (3pt);
			
			\draw[fill=black] (0,4) circle (3pt);
			\draw[fill=black] (2,4) circle (3pt);
			\draw[fill=black] (4,4) circle (3pt);
			\draw[fill=black] (6,4) circle (3pt);
			
			\draw[fill=black] (0,6) circle (3pt);
			\draw[fill=black] (2,6) circle (3pt);
			\draw[fill=black] (4,6) circle (3pt);
			\draw[fill=black] (6,6) circle (3pt);
			
			%poziome
			\draw[thick, line width=1pt, decoration={markings,
				mark=at position 0.65 with {\arrow[scale=1,>=stealth]{>>}}},
			postaction={decorate}] (0,0) -- (2,0);
			\draw[thick, line width=1pt, decoration={markings,
				mark=at position 0.65 with {\arrow[scale=1,>=stealth]{>>}}},
			postaction={decorate}] (2,0) -- (4,0);
			\draw[thick, line width=1pt, decoration={markings,
				mark=at position 0.65 with {\arrow[scale=1,>=stealth]{>>}}},
			postaction={decorate}] (4,0) -- (6,0);
			
			\draw[thick, line width=1pt, decoration={markings,
				mark=at position 0.65 with {\arrow[scale=1,>=stealth]{>>}}},
			postaction={decorate}] (0,2) -- (2,2);
			\draw[thick, line width=1pt, decoration={markings,
				mark=at position 0.65 with {\arrow[scale=1,>=stealth]{>>}}},
			postaction={decorate}] (2,2) -- (4,2);
			\draw[thick, line width=1pt, decoration={markings,
				mark=at position 0.65 with {\arrow[scale=1,>=stealth]{>>}}},
			postaction={decorate}] (4,2) -- (6,2);
			
			\draw[thick, line width=1pt, decoration={markings,
				mark=at position 0.65 with {\arrow[scale=1,>=stealth]{>>}}},
			postaction={decorate}] (0,4) -- (2,4);
			\draw[thick, line width=1pt, decoration={markings,
				mark=at position 0.65 with {\arrow[scale=1,>=stealth]{>>}}},
			postaction={decorate}] (2,4) -- (4,4);
			\draw[thick, line width=1pt, decoration={markings,
				mark=at position 0.65 with {\arrow[scale=1,>=stealth]{>>}}},
			postaction={decorate}] (4,4) -- (6,4);
			
			\draw[thick, line width=1pt, decoration={markings,
				mark=at position 0.65 with {\arrow[scale=1,>=stealth]{>>}}},
			postaction={decorate}] (0,6) -- (2,6);
			\draw[thick, line width=1pt, decoration={markings,
				mark=at position 0.65 with {\arrow[scale=1,>=stealth]{>>}}},
			postaction={decorate}] (2,6) -- (4,6);
			\draw[thick, line width=1pt, decoration={markings,
				mark=at position 0.65 with {\arrow[scale=1,>=stealth]{>>}}},
			postaction={decorate}] (4,6) -- (6,6);
			
			\draw[thick, line width=1pt, decoration={markings,
				mark=at position 0.65 with {\arrow[scale=1,>=stealth]{>>}}},
			postaction={decorate}] (0,8) -- (2,8);
			\draw[thick, line width=1pt, decoration={markings,
				mark=at position 0.65 with {\arrow[scale=1,>=stealth]{>>}}},
			postaction={decorate}] (2,8) -- (4,8);
			\draw[thick, line width=1pt, decoration={markings,
				mark=at position 0.65 with {\arrow[scale=1,>=stealth]{>>}}},
			postaction={decorate}] (4,8) -- (6,8);
			
			%pionowe
			\draw[thick, line width=1pt, decoration={markings,
				mark=at position 0.63 with {\arrow[scale=1,>=stealth]{<<}}},
			postaction={decorate}] (0,0) -- (0,2);
			\draw[thick, line width=1pt, decoration={markings,
				mark=at position 0.63 with {\arrow[scale=1,>=stealth]{<<}}},
			postaction={decorate}] (0,2) -- (0,4);
			\draw[thick, line width=1pt, decoration={markings,
				mark=at position 0.63 with {\arrow[scale=1,>=stealth]{<<}}},
			postaction={decorate}] (0,4) -- (0,6);
			\draw[thick, line width=1pt, decoration={markings,
				mark=at position 0.63 with {\arrow[scale=1,>=stealth]{<<}}},
			postaction={decorate}] (0,6) -- (0,8);
			
			\draw[thick, line width=1pt, decoration={markings,
				mark=at position 0.63 with {\arrow[scale=1,>=stealth]{<<}}},
			postaction={decorate}] (2,0) -- (2,2);
			\draw[thick, line width=1pt, decoration={markings,
				mark=at position 0.63 with {\arrow[scale=1,>=stealth]{<<}}},
			postaction={decorate}] (2,2) -- (2,4);
			\draw[thick, line width=1pt, decoration={markings,
				mark=at position 0.63 with {\arrow[scale=1,>=stealth]{<<}}},
			postaction={decorate}] (2,4) -- (2,6);
			\draw[thick, line width=1pt, decoration={markings,
				mark=at position 0.63 with {\arrow[scale=1,>=stealth]{<<}}},
			postaction={decorate}] (2,6) -- (2,8);
			
			\draw[thick, line width=1pt, decoration={markings,
				mark=at position 0.63 with {\arrow[scale=1,>=stealth]{<<}}},
			postaction={decorate}] (4,0) -- (4,2);
			\draw[thick, line width=1pt, decoration={markings,
				mark=at position 0.63 with {\arrow[scale=1,>=stealth]{<<}}},
			postaction={decorate}] (4,2) -- (4,4);
			\draw[thick, line width=1pt, decoration={markings,
				mark=at position 0.63 with {\arrow[scale=1,>=stealth]{<<}}},
			postaction={decorate}] (4,4) -- (4,6);
			\draw[thick, line width=1pt, decoration={markings,
				mark=at position 0.63 with {\arrow[scale=1,>=stealth]{<<}}},
			postaction={decorate}] (4,6) -- (4,8);
			
			\draw[thick, line width=1pt, decoration={markings,
				mark=at position 0.63 with {\arrow[scale=1,>=stealth]{<<}}},
			postaction={decorate}] (6,0) -- (6,2);
			\draw[thick, line width=1pt, decoration={markings,
				mark=at position 0.63 with {\arrow[scale=1,>=stealth]{<<}}},
			postaction={decorate}] (6,2) -- (6,4);
			\draw[thick, line width=1pt, decoration={markings,
				mark=at position 0.63 with {\arrow[scale=1,>=stealth]{<<}}},
			postaction={decorate}] (6,4) -- (6,6);
			\draw[thick, line width=1pt, decoration={markings,
				mark=at position 0.63 with {\arrow[scale=1,>=stealth]{<<}}},
			postaction={decorate}] (6,6) -- (6,8);
			
			%skosne						
			\draw[thick, line width=1pt, decoration={markings,
				mark=at position 0.6 with {\arrow[scale=1,>=stealth]{<<}}},
			postaction={decorate}] (0,2) -- (2,0);
			
			\draw[thick, line width=1pt, decoration={markings,
				mark=at position 0.6 with {\arrow[scale=1,>=stealth]{<<}}},
			postaction={decorate}] (0,4) -- (2,2);
			\draw[thick, line width=1pt, decoration={markings,
				mark=at position 0.6 with {\arrow[scale=1,>=stealth]{<<}}},
			postaction={decorate}] (2,2) -- (4,0);
			
			\draw[thick, line width=1pt, decoration={markings,
				mark=at position 0.6 with {\arrow[scale=1,>=stealth]{<<}}},
			postaction={decorate}] (0,6) -- (2,4);
			\draw[thick, line width=1pt, decoration={markings,
				mark=at position 0.6 with {\arrow[scale=1,>=stealth]{<<}}},
			postaction={decorate}] (2,4) -- (4,2);
			\draw[thick, line width=1pt, decoration={markings,
				mark=at position 0.6 with {\arrow[scale=1,>=stealth]{<<}}},
			postaction={decorate}] (4,2) -- (6,0);
			
			\draw[thick, line width=1pt, decoration={markings,
				mark=at position 0.6 with {\arrow[scale=1,>=stealth]{<<}}},
			postaction={decorate}] (2,6) -- (4,4);
			\draw[thick, line width=1pt, decoration={markings,
				mark=at position 0.6 with {\arrow[scale=1,>=stealth]{<<}}},
			postaction={decorate}] (4,4) -- (6,2);
			
			\draw[thick, line width=1pt, decoration={markings,
				mark=at position 0.6 with {\arrow[scale=1,>=stealth]{<<}}},
			postaction={decorate}] (4,6) -- (6,4);
			
			\draw[thick, line width=1pt, decoration={markings,
				mark=at position 0.6 with {\arrow[scale=1,>=stealth]{<<}}},
			postaction={decorate}] (0,8) -- (2,6);
			\draw[thick, line width=1pt, decoration={markings,
				mark=at position 0.6 with {\arrow[scale=1,>=stealth]{<<}}},
			postaction={decorate}] (2,8) -- (4,6);
			\draw[thick, line width=1pt, decoration={markings,
				mark=at position 0.6 with {\arrow[scale=1,>=stealth]{<<}}},
			postaction={decorate}] (4,8) -- (6,6);

			%%arrows
			\draw [->, line width=1pt]  (-0.3,0) -- (-0.3, 8);
			\draw [->, line width=1pt]  (6.3,0) -- (6.3, 8);
			
			\draw [->, line width=1pt]  (0,-0.3) -- (6, -0.3);
			\draw [->, line width=1pt]  (0, 8.3) -- (6, 8.3);

			\node at (0,-0.7) {$0$};
			\node at (2,-0.7) {$1$};
			\node at (4,-0.7) {$2$};
			\node at (6,-0.7) {$0$};
			
			\node at (-0.65,0) {$0$};
			\node at (-0.65,2) {$1$};
			\node at (-0.65,4) {$2$};
			\node at (-0.65,6) {$3$};
			\node at (-0.65,8) {$0$};
			
			\node at (3,-1.5) {$(\mathcal{T}_{3,4},\tau_2)$};
		\end{scope}
			
			%\draw[step=0.5, gray, thin] (-9,0) grid (15,8);
			%\draw[step=1, red, thin] (-9,-4) grid (15,8);
			%\draw[step=5, blue, thin] (-9,-4) grid (15,8);				
		\end{tikzpicture}
		\captionof{figure}{ }
	\end{center}
	We define the sets of zigzags
	$$\kappa_1=\{Z'_i:i\in\mathbb{Z}_k\},\quad\kappa_2=\{Z''_j:j\in\mathbb{Z}_m\},\quad\kappa_3=\{Z'''_s:0\leq s\leq \gcd(k,m)\}$$
	and for each $r=1,2,3$ we denote by $\kappa^{-1}_r$ the set of zigzags obtained by reversing those in $\kappa_r$. 
	Then, for the $z$-orientation $\tau_1=\kappa_1\cup\kappa_2\cup\kappa_3$, all faces of $\mathcal{T}_{k,m}$ are of type I and, for $\tau_2=\kappa_1\cup\kappa^{-1}_2\cup\kappa_3$, all faces are of type II 
	(see Fig. 4). 
}\end{exmp}

\begin{exmp}\label{ex3}{\rm 
	Consider the triangulation $\mathcal{P}$ of the real projective plane presented in Fig. 5. 
		\begin{center} 
		\begin{tikzpicture}[scale=0.65] 
				\coordinate (A1) at (90:6cm);
				\coordinate (A2) at (30:6cm);
				\coordinate (A3) at (330:6cm);
				\coordinate (a1) at (270:6cm);
				\coordinate (a2) at (210:6cm);
				\coordinate (a3) at (150:6cm);
				
				\coordinate (A4) at (30:3cm);
				\coordinate (A5) at (270:3cm);
				\coordinate (A6) at (150:3cm);
				
				\coordinate (B1) at ($ (0,0)!0.5!($(A1)+(A2)$) $);
				\coordinate (B2) at ($ (0,0)!0.5!($(A2)+(A3)$) $);
				\coordinate (B3) at ($ (0,0)!0.5!($(A3)+(a1)$) $);
				\coordinate (b1) at ($ (0,0)!0.5!($(a1)+(a2)$) $);
				\coordinate (b2) at ($ (0,0)!0.5!($(a2)+(a3)$) $);
				\coordinate (b3) at ($ (0,0)!0.5!($(a3)+(A1)$) $);
				
				\coordinate (B4) at ($ (0,0)!0.5!($(A4)+(A2)$) $);
				\coordinate (B5) at ($ (0,0)!0.5!($(A4)+(A3)$) $);
				
				\coordinate (B6) at ($ (0,0)!0.5!($(A5)+(A3)$) $);
				\coordinate (B7) at ($ (0,0)!0.5!($(A5)+(a1)$) $);
				\coordinate (B8) at ($ (0,0)!0.5!($(A5)+(a2)$) $);
				
				\coordinate (B9) at ($ (0,0)!0.5!($(A6)+(a2)$) $);
				\coordinate (B10) at ($ (0,0)!0.5!($(A6)+(a3)$) $);
				
				\coordinate (B11) at ($ (0,0)!0.5!($(A1)+(A6)$) $);
				\coordinate (B12) at ($ (0,0)!0.5!($(A1)+(A4)$) $);
				
				\coordinate (B13) at ($ (0,0)!0.5!($(A6)+(A4)$) $);
				\coordinate (B14) at ($ (0,0)!0.5!($(A5)+(A4)$) $);
				\coordinate (B15) at ($ (0,0)!0.5!($(A5)+(A6)$) $);
				
				\coordinate (C1) at ($ (0,0)!0.333333!($(A1)+(A2)+(A4)$) $);
				\coordinate (C2) at ($ (0,0)!0.333333!($(A2)+(A3)+(A4)$) $);
				\coordinate (C3) at ($ (0,0)!0.333333!($(a1)+(A3)+(A5)$) $);
				\coordinate (C4) at ($ (0,0)!0.333333!($(a1)+(a2)+(A5)$) $);
				\coordinate (C5) at ($ (0,0)!0.333333!($(a2)+(a3)+(A6)$) $);
				\coordinate (C6) at ($ (0,0)!0.333333!($(a3)+(A1)+(A6)$) $);
				\coordinate (C7) at ($ (0,0)!0.333333!($(A1)+(A4)+(A6)$) $);
				\coordinate (C8) at ($ (0,0)!0.333333!($(A3)+(A4)+(A5)$) $);
				\coordinate (C9) at ($ (0,0)!0.333333!($(a2)+(A5)+(A6)$) $);
				\coordinate (C10) at ($ (0,0)!0.333333!($(A4)+(A5)+(A6)$) $);
				
				\draw[thick, line width=0.5pt, decoration={markings,
					mark=at position 0.65 with {\arrow[scale=1,>=stealth]{>>}}},
				postaction={decorate}] (A1) -- (B1);
				\draw[thick, line width=1.25pt, decoration={markings, mark=at position 0.65 with {\arrow[scale=1,>=stealth,color=black]{>>}}},
				postaction={decorate}] (A2) -- (B1);%
				\draw[thick, line width=0.5pt, decoration={markings,
					mark=at position 0.65 with {\arrow[scale=1,>=stealth]{>>}}},
				postaction={decorate}] (A2) -- (B2);
				\draw[thick, line width=0.5pt, decoration={markings,
					mark=at position 0.65 with {\arrow[scale=1,>=stealth]{>>}}},
				postaction={decorate}] (A3) -- (B2);
				\draw[thick, line width=0.5pt, decoration={markings,
					mark=at position 0.65 with {\arrow[scale=1,>=stealth]{>>}}},
				postaction={decorate}] (A3) -- (B3);
				\draw[thick, line width=0.5pt, decoration={markings,
					mark=at position 0.65 with {\arrow[scale=1,>=stealth]{>>}}},
				postaction={decorate}] (a1) -- (B3);
				\draw[thick, line width=0.5pt, decoration={markings,
					mark=at position 0.65 with {\arrow[scale=1,>=stealth]{>>}}},
				postaction={decorate}] (a1) -- (b1);
				\draw[thick, line width=1.25pt, decoration={markings, mark=at position 0.65 with {\arrow[scale=1,>=stealth,color=black]{>>}}},
				postaction={decorate}] (a2) -- (b1);%
				\draw[thick, line width=0.5pt, decoration={markings,
					mark=at position 0.65 with {\arrow[scale=1,>=stealth]{>>}}},
				postaction={decorate}] (a2) -- (b2);
				\draw[thick, line width=0.5pt, decoration={markings,
					mark=at position 0.65 with {\arrow[scale=1,>=stealth]{>>}}},
				postaction={decorate}] (a3) -- (b2);
				\draw[thick, line width=0.5pt, decoration={markings,
					mark=at position 0.65 with {\arrow[scale=1,>=stealth]{>>}}},
				postaction={decorate}] (a3) -- (b3);
				\draw[thick, line width=0.5pt, decoration={markings,
					mark=at position 0.65 with {\arrow[scale=1,>=stealth]{>>}}},
				postaction={decorate}] (A1) -- (b3);
				
				\draw[thick, line width=0.5pt, decoration={markings,
					mark=at position 0.65 with {\arrow[scale=1,>=stealth]{>>}}},
				postaction={decorate}] (A4) -- (B14);
				\draw[thick, line width=0.5pt, decoration={markings,
					mark=at position 0.65 with {\arrow[scale=1,>=stealth]{>>}}},
				postaction={decorate}] (A5) -- (B14);
				\draw[thick, line width=0.5pt, decoration={markings,
					mark=at position 0.65 with {\arrow[scale=1,>=stealth]{>>}}},
				postaction={decorate}] (A5) -- (B15);
				\draw[thick, line width=1.25pt, decoration={markings, mark=at position 0.65 with {\arrow[scale=1,>=stealth,color=black]{>>}}},
				postaction={decorate}] (A6) -- (B15);% 		
				\draw[thick, line width=0.5pt, decoration={markings,
					mark=at position 0.65 with {\arrow[scale=1,>=stealth]{>>}}},
				postaction={decorate}] (A6) -- (B13);
				\draw[thick, line width=1.25pt, decoration={markings, mark=at position 0.65 with {\arrow[scale=1,>=stealth,color=black]{>>}}},
				postaction={decorate}] (A4) -- (B13);%
				
				\draw[thick, line width=1.25pt, decoration={markings, mark=at position 0.65 with {\arrow[scale=1,>=stealth,color=black]{>>}}},
				postaction={decorate}] (A1) -- (B12);%
				\draw[thick, line width=0.5pt, decoration={markings,
					mark=at position 0.65 with {\arrow[scale=1,>=stealth]{>>}}},
				postaction={decorate}] (A4) -- (B12);
				\draw[thick, line width=0.5pt, decoration={markings,
					mark=at position 0.65 with {\arrow[scale=1,>=stealth]{>>}}},
				postaction={decorate}] (A1) -- (B11);
				\draw[thick, line width=0.5pt, decoration={markings,
					mark=at position 0.65 with {\arrow[scale=1,>=stealth]{>>}}},
				postaction={decorate}] (A6) -- (B11);
				
				\draw[thick, line width=0.5pt, decoration={markings,
					mark=at position 0.65 with {\arrow[scale=1,>=stealth]{>>}}},
				postaction={decorate}] (A3) -- (B5);
				\draw[thick, line width=0.5pt, decoration={markings,
					mark=at position 0.65 with {\arrow[scale=1,>=stealth]{>>}}},
				postaction={decorate}] (A4) -- (B5);
				\draw[thick, line width=0.5pt, decoration={markings,
					mark=at position 0.65 with {\arrow[scale=1,>=stealth]{>>}}},
				postaction={decorate}] (A3) -- (B6);
				\draw[thick, line width=0.5pt, decoration={markings,
					mark=at position 0.65 with {\arrow[scale=1,>=stealth]{>>}}},
				postaction={decorate}] (A5) -- (B6);
				
				\draw[thick, line width=0.5pt, decoration={markings,
					mark=at position 0.65 with {\arrow[scale=1,>=stealth]{>>}}},
				postaction={decorate}] (a2) -- (B8);
				\draw[thick, line width=1.25pt, decoration={markings, mark=at position 0.65 with {\arrow[scale=1,>=stealth,color=black]{>>}}},
				postaction={decorate}] (A5) -- (B8);%
				\draw[thick, line width=0.5pt, decoration={markings,
					mark=at position 0.65 with {\arrow[scale=1,>=stealth]{>>}}},
				postaction={decorate}] (a2) -- (B9);
				\draw[thick, line width=0.5pt, decoration={markings,
					mark=at position 0.65 with {\arrow[scale=1,>=stealth]{>>}}},
				postaction={decorate}] (A6) -- (B9);
				
				\draw[thick, line width=0.5pt, decoration={markings,
					mark=at position 0.65 with {\arrow[scale=1,>=stealth]{>>}}},
				postaction={decorate}] (A2) -- (B4);
				\draw[thick, line width=0.5pt, decoration={markings,
					mark=at position 0.65 with {\arrow[scale=1,>=stealth]{>>}}},
				postaction={decorate}] (A4) -- (B4);
				\draw[thick, line width=0.5pt, decoration={markings,
					mark=at position 0.7 with {\arrow[scale=1,>=stealth]{>>}}},
				postaction={decorate}] (a1) -- (B7);
				\draw[thick, line width=0.5pt, decoration={markings,
					mark=at position 0.7 with {\arrow[scale=1,>=stealth]{>>}}},
				postaction={decorate}] (A5) -- (B7);
				\draw[thick, line width=0.5pt, decoration={markings,
					mark=at position 0.65 with {\arrow[scale=1,>=stealth]{>>}}},
				postaction={decorate}] (a3) -- (B10);
				\draw[thick, line width=0.5pt, decoration={markings,
					mark=at position 0.65 with {\arrow[scale=1,>=stealth]{>>}}},
				postaction={decorate}] (A6) -- (B10);
				
				\draw[thick, line width=1.25pt, decoration={markings, mark=at position 0.65 with {\arrow[scale=1,>=stealth,color=black]{>>}}},
				postaction={decorate}] (C1) -- (A1);%
				\draw[thick, line width=0.5pt, decoration={markings,
					mark=at position 0.65 with {\arrow[scale=1,>=stealth]{>>}}},
				postaction={decorate}] (C1) -- (A2);
				\draw[thick, line width=0.5pt, decoration={markings,
					mark=at position 0.65 with {\arrow[scale=1,>=stealth]{>>}}},
				postaction={decorate}] (C1) -- (A4);
				\draw[thick, line width=1.25pt, decoration={markings, mark=at position 0.8 with {\arrow[scale=1,>=stealth,color=black]{>>}}},
				postaction={decorate}] (B1) -- (C1);%
				\draw[thick, line width=0.5pt, decoration={markings,
					mark=at position 0.65 with {\arrow[scale=1,>=stealth]{>>}}},
				postaction={decorate}] (B4) -- (C1);
				\draw[thick, line width=0.5pt, decoration={markings,
					mark=at position 0.65 with {\arrow[scale=1,>=stealth]{>>}}},
				postaction={decorate}] (B12) -- (C1);
				
				\draw[thick, line width=0.5pt, decoration={markings,
					mark=at position 0.65 with {\arrow[scale=1,>=stealth]{>>}}},
				postaction={decorate}] (C2) -- (A2);
				\draw[thick, line width=0.5pt, decoration={markings,
					mark=at position 0.65 with {\arrow[scale=1,>=stealth]{>>}}},
				postaction={decorate}] (C2) -- (A3);
				\draw[thick, line width=0.5pt, decoration={markings,
					mark=at position 0.65 with {\arrow[scale=1,>=stealth]{>>}}},
				postaction={decorate}] (C2) -- (A4);
				\draw[thick, line width=0.5pt, decoration={markings,
					mark=at position 0.65 with {\arrow[scale=1,>=stealth]{>>}}},
				postaction={decorate}] (B2) -- (C2);
				\draw[thick, line width=0.5pt, decoration={markings,
					mark=at position 0.65 with {\arrow[scale=1,>=stealth]{>>}}},
				postaction={decorate}] (B4) -- (C2);
				\draw[thick, line width=0.5pt, decoration={markings,
					mark=at position 0.65 with {\arrow[scale=1,>=stealth]{>>}}},
				postaction={decorate}] (B5) -- (C2);
				
				\draw[thick, line width=0.5pt, decoration={markings,
					mark=at position 0.6 with {\arrow[scale=1,>=stealth]{>>}}},
				postaction={decorate}] (C3) -- (a1);
				\draw[thick, line width=0.5pt, decoration={markings,
					mark=at position 0.55 with {\arrow[scale=1,>=stealth]{>>}}},
				postaction={decorate}] (C3) -- (A3);
				\draw[thick, line width=0.5pt, decoration={markings,
					mark=at position 0.6 with {\arrow[scale=1,>=stealth]{>>}}},
				postaction={decorate}] (C3) -- (A5);
				\draw[thick, line width=0.5pt, decoration={markings,
					mark=at position 0.75 with {\arrow[scale=1,>=stealth]{>>}}},
				postaction={decorate}] (B3) -- (C3);
				\draw[thick, line width=0.5pt, decoration={markings,
					mark=at position 0.65 with {\arrow[scale=1,>=stealth]{>>}}},
				postaction={decorate}] (B6) -- (C3);
				\draw[thick, line width=0.5pt, decoration={markings,
					mark=at position 0.65 with {\arrow[scale=1,>=stealth]{>>}}},
				postaction={decorate}] (B7) -- (C3);
				
				\draw[thick, line width=0.5pt, decoration={markings,
					mark=at position 0.6 with {\arrow[scale=1,>=stealth]{>>}}},
				postaction={decorate}] (C4) -- (a1);
				\draw[thick, line width=1.25pt, decoration={markings, mark=at position 0.65 with {\arrow[scale=1,>=stealth,color=black]{>>}}},
				postaction={decorate}] (C4) -- (a2);%
				\draw[thick, line width=0.5pt, decoration={markings,
					mark=at position 0.6 with {\arrow[scale=1,>=stealth]{>>}}},
				postaction={decorate}] (C4) -- (A5);
				\draw[thick, line width=0.5pt, decoration={markings,
					mark=at position 0.75 with {\arrow[scale=1,>=stealth]{>>}}},
				postaction={decorate}] (b1) -- (C4);
				\draw[thick, line width=0.5pt, decoration={markings,
					mark=at position 0.65 with {\arrow[scale=1,>=stealth]{>>}}},
				postaction={decorate}] (B7) -- (C4);
				\draw[thick, line width=1.25pt, decoration={markings, mark=at position 0.65 with {\arrow[scale=1,>=stealth,color=black]{>>}}},
				postaction={decorate}] (B8) -- (C4);%
				
				\draw[thick, line width=0.5pt, decoration={markings,
					mark=at position 0.65 with {\arrow[scale=1,>=stealth]{>>}}},
				postaction={decorate}] (C5) -- (a2);
				\draw[thick, line width=0.5pt, decoration={markings,
					mark=at position 0.65 with {\arrow[scale=1,>=stealth]{>>}}},
				postaction={decorate}] (C5) -- (a3);
				\draw[thick, line width=0.5pt, decoration={markings,
					mark=at position 0.65 with {\arrow[scale=1,>=stealth]{>>}}},
				postaction={decorate}] (C5) -- (A6);
				\draw[thick, line width=0.5pt, decoration={markings,
					mark=at position 0.65 with {\arrow[scale=1,>=stealth]{>>}}},
				postaction={decorate}] (b2) -- (C5);
				\draw[thick, line width=0.5pt, decoration={markings,
					mark=at position 0.65 with {\arrow[scale=1,>=stealth]{>>}}},
				postaction={decorate}] (B9) -- (C5);
				\draw[thick, line width=0.5pt, decoration={markings,
					mark=at position 0.65 with {\arrow[scale=1,>=stealth]{>>}}},
				postaction={decorate}] (B10) -- (C5);
				
				\draw[thick, line width=0.5pt, decoration={markings,
					mark=at position 0.65 with {\arrow[scale=1,>=stealth]{>>}}},
				postaction={decorate}] (C6) -- (A1);
				\draw[thick, line width=0.5pt, decoration={markings,
					mark=at position 0.65 with {\arrow[scale=1,>=stealth]{>>}}},
				postaction={decorate}] (C6) -- (a3);
				\draw[thick, line width=0.5pt, decoration={markings,
					mark=at position 0.65 with {\arrow[scale=1,>=stealth]{>>}}},
				postaction={decorate}] (C6) -- (A6);
				\draw[thick, line width=0.5pt, decoration={markings,
					mark=at position 0.65 with {\arrow[scale=1,>=stealth]{>>}}},
				postaction={decorate}] (b3) -- (C6);
				\draw[thick, line width=0.5pt, decoration={markings,
					mark=at position 0.65 with {\arrow[scale=1,>=stealth]{>>}}},
				postaction={decorate}] (B10) -- (C6);
				\draw[thick, line width=0.5pt, decoration={markings,
					mark=at position 0.65 with {\arrow[scale=1,>=stealth]{>>}}},
				postaction={decorate}] (B11) -- (C6);
				
				\draw[thick, line width=0.5pt, decoration={markings,
					mark=at position 0.65 with {\arrow[scale=1,>=stealth]{>>}}},
				postaction={decorate}] (C7) -- (A1);
				\draw[thick, line width=1.25pt, decoration={markings, mark=at position 0.65 with {\arrow[scale=1,>=stealth,color=black]{>>}}},
				postaction={decorate}] (C7) -- (A4);%
				\draw[thick, line width=0.5pt, decoration={markings,
					mark=at position 0.65 with {\arrow[scale=1,>=stealth]{>>}}},
				postaction={decorate}] (C7) -- (A6);
				\draw[thick, line width=0.5pt, decoration={markings,
					mark=at position 0.65 with {\arrow[scale=1,>=stealth]{>>}}},
				postaction={decorate}] (B11) -- (C7);
				\draw[thick, line width=1.25pt, decoration={markings, mark=at position 0.65 with {\arrow[scale=1,>=stealth,color=black]{>>}}},
				postaction={decorate}] (B12) -- (C7);%
				\draw[thick, line width=0.5pt, decoration={markings,
					mark=at position 0.65 with {\arrow[scale=1,>=stealth]{>>}}},
				postaction={decorate}] (B13) -- (C7);
				
				\draw[thick, line width=0.5pt, decoration={markings,
					mark=at position 0.65 with {\arrow[scale=1,>=stealth]{>>}}},
				postaction={decorate}] (C8) -- (A3);
				\draw[thick, line width=0.5pt, decoration={markings,
					mark=at position 0.65 with {\arrow[scale=1,>=stealth]{>>}}},
				postaction={decorate}] (C8) -- (A4);
				\draw[thick, line width=0.5pt, decoration={markings,
					mark=at position 0.65 with {\arrow[scale=1,>=stealth]{>>}}},
				postaction={decorate}] (C8) -- (A5);
				\draw[thick, line width=0.5pt, decoration={markings,
					mark=at position 0.65 with {\arrow[scale=1,>=stealth]{>>}}},
				postaction={decorate}] (B5) -- (C8);
				\draw[thick, line width=0.5pt, decoration={markings,
					mark=at position 0.65 with {\arrow[scale=1,>=stealth]{>>}}},
				postaction={decorate}] (B6) -- (C8);
				\draw[thick, line width=0.5pt, decoration={markings,
					mark=at position 0.65 with {\arrow[scale=1,>=stealth]{>>}}},
				postaction={decorate}] (B14) -- (C8);
				
				\draw[thick, line width=0.5pt, decoration={markings,
					mark=at position 0.65 with {\arrow[scale=1,>=stealth]{>>}}},
				postaction={decorate}] (C9) -- (a2);
				\draw[thick, line width=1.25pt, decoration={markings, mark=at position 0.65 with {\arrow[scale=1,>=stealth,color=black]{>>}}},
				postaction={decorate}] (C9) -- (A5);%
				\draw[thick, line width=0.5pt, decoration={markings,
					mark=at position 0.65 with {\arrow[scale=1,>=stealth]{>>}}},
				postaction={decorate}] (C9) -- (A6);
				\draw[thick, line width=0.5pt, decoration={markings,
					mark=at position 0.65 with {\arrow[scale=1,>=stealth]{>>}}},
				postaction={decorate}] (B8) -- (C9);
				\draw[thick, line width=0.5pt, decoration={markings,
					mark=at position 0.65 with {\arrow[scale=1,>=stealth]{>>}}},
				postaction={decorate}] (B9) -- (C9);
				\draw[thick, line width=1.25pt, decoration={markings, mark=at position 0.65 with {\arrow[scale=1,>=stealth,color=black]{>>}}},
				postaction={decorate}] (B15) -- (C9);%
				
				\draw[thick, line width=0.5pt, decoration={markings,
					mark=at position 0.65 with {\arrow[scale=1,>=stealth]{>>}}},
				postaction={decorate}] (C10) -- (A4);
				\draw[thick, line width=0.5pt, decoration={markings,
					mark=at position 0.65 with {\arrow[scale=1,>=stealth]{>>}}},
				postaction={decorate}] (C10) -- (A5);
				\draw[thick, line width=1.25pt, decoration={markings, mark=at position 0.65 with {\arrow[scale=1,>=stealth,color=black]{>>}}},
				postaction={decorate}] (C10) -- (A6);%
				\draw[thick, line width=1.25pt, decoration={markings, mark=at position 0.65 with {\arrow[scale=1,>=stealth,color=black]{>>}}},postaction={decorate}] (B13) -- (C10);%
				\draw[thick, line width=0.5pt, decoration={markings,
					mark=at position 0.65 with {\arrow[scale=1,>=stealth]{>>}}},
				postaction={decorate}] (B14) -- (C10);
				\draw[thick, line width=0.5pt, decoration={markings,
					mark=at position 0.65 with {\arrow[scale=1,>=stealth]{>>}}},
				postaction={decorate}] (B15) -- (C10);
				
				%\draw [->, line width=1pt]  (27:6.6cm) -- (333:6.6cm);
				%\draw [->, line width=1pt]  (87:6.6cm) -- (33:6.6cm);
				%\draw [->, line width=1pt]  (147:6.6cm) -- (93:6.6cm);
				%\draw [->, line width=1pt]  (207:6.6cm) -- (153:6.6cm);
				%\draw [->, line width=1pt]  (267:6.6cm) -- (213:6.6cm);
				%\draw [->, line width=1pt]  (327:6.6cm) -- (273:6.6cm);
				
				\draw[fill=black] (A1) circle (2.5pt);
				\draw[fill=black] (A2) circle (2.5pt);
				\draw[fill=black] (A3) circle (2.5pt);
				\draw[fill=black] (a1) circle (2.5pt);
				\draw[fill=black] (a2) circle (2.5pt);
				\draw[fill=black] (a3) circle (2.5pt);
				\draw[fill=black] (A4) circle (2.5pt);
				\draw[fill=black] (A5) circle (2.5pt);
				\draw[fill=black] (A6) circle (2.5pt);
				
				\draw[fill=black] (B1) circle (2.5pt);
				\draw[fill=black] (B2) circle (2.5pt);
				\draw[fill=black] (B3) circle (2.5pt);
				\draw[fill=black] (b1) circle (2.5pt);
				\draw[fill=black] (b2) circle (2.5pt);
				\draw[fill=black] (b3) circle (2.5pt);
				\draw[fill=black] (B4) circle (2.5pt);
				\draw[fill=black] (B5) circle (2.5pt);
				\draw[fill=black] (B6) circle (2.5pt);
				\draw[fill=black] (B7) circle (2.5pt);
				\draw[fill=black] (B8) circle (2.5pt);
				\draw[fill=black] (B9) circle (2.5pt);
				\draw[fill=black] (B10) circle (2.5pt);
				\draw[fill=black] (B11) circle (2.5pt);
				\draw[fill=black] (B12) circle (2.5pt);
				\draw[fill=black] (B13) circle (2.5pt);
				\draw[fill=black] (B14) circle (2.5pt);
				\draw[fill=black] (B15) circle (2.5pt);
				
				\draw[fill=black] (C1) circle (2.5pt);
				\draw[fill=black] (C2) circle (2.5pt);
				\draw[fill=black] (C3) circle (2.5pt);
				\draw[fill=black] (C4) circle (2.5pt);
				\draw[fill=black] (C5) circle (2.5pt);
				\draw[fill=black] (C6) circle (2.5pt);
				\draw[fill=black] (C7) circle (2.5pt);
				\draw[fill=black] (C8) circle (2.5pt);
				\draw[fill=black] (C9) circle (2.5pt);
				\draw[fill=black] (C10) circle (2.5pt); 
				
				\node at (30:6.4cm) {$a_2$};
				\node at (90:6.35cm) {$a_1$};
				\node at (150:6.4cm) {$a_3$};
				\node at (210:6.4cm) {$a_2$};
				\node at (270:6.35cm) {$a_1$};
				\node at (330:6.4cm) {$a_3$};
				
				\node at (0:5.55cm) {$b_2$};
				\node at (60:5.55cm) {$b_1$};
				\node at (120:5.6cm) {$b_3$};
				\node at (180:5.55cm) {$b_2$};
				\node at (240:5.55cm) {$b_1$};
				\node at (300:5.6cm) {$b_3$};
				
			%	\iffalse
				\node at (2.88cm,2.02cm) {$a_4$};
				\node at (-2.89cm,1.95cm) {$a_6$};
				\node at (0.28cm,-3.41cm) {$a_5$};
				
				\node at (4.26cm,2.08cm) {$b_4$};
				\node at (-4.3cm,2.15cm) {$b_{10}$}; 
				\node at (3.87cm,-1.16cm) {$b_5$};
				\node at (-3.74cm,-1.1cm) {$b_9$};
				\node at (2.92cm,-2.72cm) {$b_6$};
				\node at (-2.88cm,-2.72cm) {$b_8$};
				\node at (0.3cm,-4.75cm) {$b_7$};
				\node at (-0.76cm,3.85cm) {$b_{11}$}; 
				\node at (0.84cm,3.92cm) {$b_{12}$}; 
				\node at (0.44cm,1.78cm) {$b_{13}$}; 
				\node at (1.85cm,-0.65cm) {$b_{14}$}; 
				\node at (-1.75cm,-0.57cm) {$b_{15}$}; 
				
				\node at (2.37cm,3.32cm) {$c_1$};
				\node at (-2.3cm,3.3cm) {$c_6$};
				\node at (4.7cm,0.53cm) {$c_2$};
				\node at (-4.65cm,0.6cm) {$c_5$};
				\node at (1.63cm,-3.72cm) {$c_3$};
				\node at (-1.7cm,-3.67cm) {$c_4$};
				\node at (0.43cm,2.96cm) {$c_7$}; 
				\node at (2.865cm,-1.06cm) {$c_8$};
				\node at (-2.85cm,-1.1cm) {$c_9$};
				\node at (0.55cm,-0.028cm) {$c_{10}$}; 
			%	\fi

		\end{tikzpicture}
		\captionof{figure}{ }
	\end{center}
There are $12$ zigzags up to reversing (one of them is marked with a thick line in Fig. 5) and $\mathcal{P}$ admits the $z$-orientation $\tau_1$ with all faces of type II (shown in Fig. 5).  
A direct verification shows that each edge of $\mathcal{P}$ occurs at most once in any zigzag. 
Therefore, by replacing one of zigzags in $\tau_1$ with its reversed zigzag, we obtain a new $z$-orientation $\tau_2$ where some faces are of type I. 
Note that $\mathcal{P}$ admits a $3$-coloring with the subsets $\{a_i\}, \{b_i\}, \{c_i\}$ corresponding to different colors.
}\end{exmp}

%%%%%%%%%%%%%%%%%%%%%%%%%%%%%%%%%%%%%%%%%%%%%%%%%%%%%%%%%%%%%%%%%%%%%%%%%%%%%%%%%%%%%%%%%%%%%%%
\section{Markov chains of $z$-oriented triangulations}
Consider a $z$-oriented triangulation $(\Gamma,\tau)$ whose vertex set is $V=\{v_1,\dots,v_n\}$. 
We define a Markov chain corresponding to $(\Gamma,\tau)$ by constructing its transition graph $\Gamma_{\tau}$. 

We replace each edge of type I with two oppositely directed edges and each edge of type II with two edges directed in the same way as the original edge. 
The obtained digraph will be denoted by $\Gamma'$. 
Let $\mathrm{outdeg}_{\Gamma'}(v)$ be the outdegree of a vertex $v\in V$ in $\Gamma'$. 
For every edge $vw$ of $\Gamma'$ directed from $v$ to $w$, we assign the number $\frac{1}{\mathrm{outdeg}_{\Gamma'}(v)}$, i.e. the probability of leaving $v$ through $vw$ when all outgoing edges are chosen with equal probability. 

Since each edge $vw$ of type II in $\Gamma$ is replaced by two edges directed from $v$ to $w$ in $\Gamma'$, it is natural to restore the directed edge $vw$ in their place and assign to this edge the sum of their probabilities $\frac{2}{\mathrm{outdeg}_{\Gamma'}(v)}$. 
As a result, we obtain the digraph $\Gamma_{\tau}$ with the vertex set $V$ and the probability $p_{ij}$ assigned to every directed edge $v_iv_j$ such that 
$$p_{ij}=
\begin{cases}
	\frac{1}{\mathrm{outdeg}_{\Gamma'}(v_i)} & \text{if } v_i\text{ and }v_j\text{ form an edge of type I in }\Gamma, \\
	\frac{2}{\mathrm{outdeg}_{\Gamma'}(v_i)} & \text{if } v_i\text{ and }v_j\text{ form an edge of type II in }\Gamma\text{ directed from }v_i\text{ to }v_j, \\
	0 & \text{otherwise}, 
\end{cases}$$
where $i,j\in\{1,\dots,n\}$, see Fig. 6. 
Note that $p_{ij}=0$ if there is no edge joining $v_i$ and $v_j$ or there is the edge of type II directed from $v_j$ to $v_i$. 
\begin{center}
	\begin{tikzpicture}[scale=0.6]
		
		%U1
		\draw[fill=black] (-1,0) circle (3pt);
		\draw[fill=black] (2,0) circle (3pt);
	
		\draw [thick, line width=1pt, decoration={markings,
			mark=at position 0.62 with {\arrow[scale=1.25,>=stealth]{><}}},
			postaction={decorate}] (-1,0) -- (2,0);
		
		\node at (-1,-0.5) {$v_i$};
		\node at (2,-0.5) {$v_j$};
		
		%U2
		\draw[fill=black] (5,0) circle (3pt);
		\draw[fill=black] (8,0) circle (3pt);
		
		\draw [thick, line width=1pt, decoration={markings,
			mark=at position 0.56 with {\arrow[scale=1.25,>=stealth]{>}}},
			postaction={decorate}] (5,0) to[bend left=30] (8,0);
		
		\draw [thick, line width=1pt, decoration={markings,
			mark=at position 0.56 with {\arrow[scale=1.25,>=stealth]{<}}},
			postaction={decorate}] (5,0) to[bend right=30] (8,0);
				
		\node at (5,-0.5) {$v_i$};
		\node at (8,-0.5) {$v_j$};
		
		\node at (6.5,1.1) {\small $\frac{1}{\mathrm{outdeg}_{\Gamma'}(v_i)}$};
		\node at (6.5,-1.1) {\small $\frac{1}{\mathrm{outdeg}_{\Gamma'}(v_j)}$};
				
		%U3
		\draw[fill=black] (11,0) circle (3pt);
		\draw[fill=black] (14,0) circle (3pt);

		\draw [thick, line width=1pt, decoration={markings,
			mark=at position 0.56 with {\arrow[scale=1.25,>=stealth]{>}}},
			postaction={decorate}] (11,0) to[bend left=30] (14,0);

		\draw [thick, line width=1pt, decoration={markings,
			mark=at position 0.56 with {\arrow[scale=1.25,>=stealth]{<}}},
			postaction={decorate}] (11,0) to[bend right=30] (14,0);

		\node at (11,-0.5) {$v_i$};
		\node at (14,-0.5) {$v_j$};
		
		\node at (12.5,1.1) {\small $\frac{1}{\mathrm{outdeg}_{\Gamma'}(v_i)}$};
		\node at (12.5,-1.1) {\small $\frac{1}{\mathrm{outdeg}_{\Gamma'}(v_j)}$};
		
		%L1
		\draw[fill=black] (-1,-4) circle (3pt);
		\draw[fill=black] (2,-4) circle (3pt);
		
		\draw [thick, line width=1pt, decoration={markings,
			mark=at position 0.62 with {\arrow[scale=1.25,>=stealth]{>>}}},
			postaction={decorate}] (-1,-4) -- (2,-4);
		
		\node at (-1,-4.5) {$v_i$};
		\node at (2,-4.5) {$v_j$};
		
		%L2
		\draw[fill=black] (5,-4) circle (3pt);
		\draw[fill=black] (8,-4) circle (3pt);

		\draw [thick, line width=1pt, decoration={markings,
			mark=at position 0.56 with {\arrow[scale=1.25,>=stealth]{>}}},
			postaction={decorate}] (5,-4) to[bend left=30] (8,-4);

		\draw [thick, line width=1pt, decoration={markings,
			mark=at position 0.56 with {\arrow[scale=1.25,>=stealth]{>}}},
			postaction={decorate}] (5,-4) to[bend right=30] (8,-4);

		\node at (5,-4.5) {$v_i$};
		\node at (8,-4.5) {$v_j$};

		\node at (6.5,-2.9) {\small $\frac{1}{\mathrm{outdeg}_{\Gamma'}(v_i)}$};
		\node at (6.5,-5.1) {\small $\frac{1}{\mathrm{outdeg}_{\Gamma'}(v_i)}$};
		
		%U3
		\draw[fill=black] (11,-4) circle (3pt);
		\draw[fill=black] (14,-4) circle (3pt);
		
		\draw [thick, line width=1pt, decoration={markings,
			mark=at position 0.56 with {\arrow[scale=1.25,>=stealth]{>}}},
			postaction={decorate}] (11,-4) -- (14,-4);
			
		\node at (12.5,-3.4) {\small $\frac{2}{\mathrm{outdeg}_{\Gamma'}(v_i)}$};
		
		\node at (11,-4.5) {$v_i$};
		\node at (14,-4.5) {$v_j$};
		
		%graphs names
		\node at (0.5,-6.5) {$\Gamma$};
		\node at (6.5,-6.5) {$\Gamma'$};
		\node at (12.5,-6.5) {$\Gamma_{\tau}$};

		\draw[->] (2.75,0) -- (4.25,0);
		\draw[->] (2.75,-4) -- (4.25,-4);
		
		\draw[->] (8.75,0) -- (10.25,0);
		\draw[->] (8.75,-4) -- (10.25,-4);
		
	\end{tikzpicture}
	\captionof{figure}{ }
\end{center}
If all faces of $(\Gamma,\tau)$ are of type II, then $\Gamma_{\tau}=\Gamma_{II}$ is a directed Eulerian triangulation of $M$. 

We denote by $\mathcal{X}_{\tau}=\{X_i\}_{i\in\mathbb{N}}$ the time-homogeneous Markov chain of $(\Gamma,\tau)$, i.e. the Markov chain whose state space is $V=\{v_1,\dots,v_n\}$, the transition graph is $\Gamma_{\tau}$ and the transition matrix is $[p_{ij}]$. 
Recall that the length of a walk is the number of edges it passes through, counting repeated edges as many times as they occur in the walk. 
\begin{prop}\label{prop1}
	The following assertions are fulfilled: 
	\begin{enumerate}
		\item[$(1)$] For every vertex of $\Gamma_{\tau}$ there exists a closed walk of length $3$ passing through this vertex.  
		\item[$(2)$] $\mathcal{X}_{\tau}$ is irreducible. 
	\end{enumerate}
\end{prop}
\begin{proof} 
	$(1)$. Let $v_1\in V$ and consider any face $F$ of $\Gamma$ containing $v_1$. 
	Suppose that $v_2, v_3$ are the other vertices of $F$. 
	Since faces of both types contain an edge of type II, we assume that there is an edge of type II directed from $v_1$ to $v_2$ (see Fig. 7). 
	If this assumption does not hold, then one of the other edges is of type II and the proof is similar. 
		\begin{center}
		\begin{tikzpicture}[scale=0.6]
			
			%U1
			\draw[fill=black] (-1,2) circle (3pt);
			\draw[fill=black] (-2.7320508076,-1) circle (3pt);
			\draw[fill=black] (0.7320508076,-1) circle (3pt);
			
			\draw [thick, line width=1pt, decoration={markings,
				mark=at position 0.62 with {\arrow[scale=1.25,>=stealth]{><}}},
			postaction={decorate}] (-1,2) -- (-2.7320508076,-1);
			
			\draw [thick, line width=1pt, decoration={markings,
				mark=at position 0.62 with {\arrow[scale=1.25,>=stealth]{>>}}},
			postaction={decorate}] (-2.7320508076,-1) -- (0.7320508076,-1);
			
			\draw [thick, line width=1pt, decoration={markings,
				mark=at position 0.62 with {\arrow[scale=1.25,>=stealth]{><}}},
			postaction={decorate}] (-1,2) -- (0.7320508076,-1);
			
			\node at (-1,0) {$F$};
			
			\node at (-2.7320508076,-1.4) {$v_1$};
			\node at (0.7320508076,-1.4) {$v_2$};
			\node at (-1,2.4) {$v_3$};
			
			%U2
			\draw[fill=black] (6.1961524228,2) circle (3pt);
			\draw[fill=black] (4.4641016152,-1) circle (3pt);
			\draw[fill=black] (7.9282032304,-1) circle (3pt);
			
			\draw [thick, line width=1pt, decoration={markings,
				mark=at position 0.56 with {\arrow[scale=1.25,>=stealth]{>}}},
			postaction={decorate}] (6.1961524228,2) to[bend left=20] (4.4641016152,-1);
			
			\draw [thick, line width=1pt, decoration={markings,
				mark=at position 0.56 with {\arrow[scale=1.25,>=stealth]{>}}},
			postaction={decorate}] (4.4641016152,-1) to[bend left=20] (6.1961524228,2);
			
			\draw [thick, line width=1pt, decoration={markings,
				mark=at position 0.56 with {\arrow[scale=1.25,>=stealth]{>}}},
			postaction={decorate}] (4.4641016152,-1) -- (7.9282032304,-1);
			
			\draw [thick, line width=1pt, decoration={markings,
				mark=at position 0.56 with {\arrow[scale=1.25,>=stealth]{>}}},
			postaction={decorate}] (7.9282032304,-1) to[bend right=20] (6.1961524228,2);
			
			\draw [thick, line width=1pt, decoration={markings,
				mark=at position 0.56 with {\arrow[scale=1.25,>=stealth]{>}}},
			postaction={decorate}] (6.1961524228,2) to[bend right=20] (7.9282032304,-1);
			
			\draw[->] (1.4820508076,0.5) -- (3.7141016152,0.5);
			
			\node at (4.4641016152,-1.4) {$v_1$};
			\node at (7.9282032304,-1.4) {$v_2$};
			\node at (6.1961524228,2.4) {$v_3$};
			
			%L1
			\draw[fill=black] (-1,-3) circle (3pt);
			\draw[fill=black] (-2.7320508076,-6) circle (3pt);
			\draw[fill=black] (0.7320508076,-6) circle (3pt);
			
			\draw [thick, line width=1pt, decoration={markings,
				mark=at position 0.62 with {\arrow[scale=1.25,>=stealth]{>>}}},
			postaction={decorate}] (-1,-3) -- (-2.7320508076,-6);
			
			\draw [thick, line width=1pt, decoration={markings,
				mark=at position 0.62 with {\arrow[scale=1.25,>=stealth]{>>}}},
			postaction={decorate}] (-2.7320508076,-6) -- (0.7320508076,-6);
			
			\draw [thick, line width=1pt, decoration={markings,
				mark=at position 0.62 with {\arrow[scale=1.25,>=stealth]{<<}}},
			postaction={decorate}] (-1,-3) -- (0.7320508076,-6);
			
			\node at (-1,-5) {$F$};
			
			\node at (-1,-7.5) {$\Gamma$};
			
			\node at (-2.7320508076,-6.4) {$v_1$};
			\node at (0.7320508076,-6.4) {$v_2$};
			\node at (-1,-2.6) {$v_3$};
			
			%L2
			\draw[fill=black] (6.1961524228,-3) circle (3pt);
			\draw[fill=black] (4.4641016152,-6) circle (3pt);
			\draw[fill=black] (7.9282032304,-6) circle (3pt);
			
			\draw [thick, line width=1pt, decoration={markings,
				mark=at position 0.56 with {\arrow[scale=1.25,>=stealth]{>}}},
			postaction={decorate}] (6.1961524228,-3) -- (4.4641016152,-6);
			
			\draw [thick, line width=1pt, decoration={markings,
				mark=at position 0.56 with {\arrow[scale=1.25,>=stealth]{>}}},
			postaction={decorate}] (4.4641016152,-6) -- (7.9282032304,-6);
			
			\draw [thick, line width=1pt, decoration={markings,
				mark=at position 0.56 with {\arrow[scale=1.25,>=stealth]{<}}},
			postaction={decorate}] (6.1961524228,-3) -- (7.9282032304,-6);
			
			\node at (6.1961524228,-7.5) {$\Gamma_{\tau}$};
			
			\draw[->] (1.4820508076,-4.5) -- (3.7141016152,-4.5);
			
			\node at (4.4641016152,-6.4) {$v_1$};
			\node at (7.9282032304,-6.4) {$v_2$};
			\node at (6.1961524228,-2.6) {$v_3$};
			
		\end{tikzpicture}
		\captionof{figure}{ }
	\end{center}
	If $F$ is of type I, then the other edges are of type I and $F$ corresponds to a subgraph of $\Gamma_{\tau}$ whose directed edges are $v_1v_2, v_2v_3, v_3v_2, v_3v_1, v_1v_3$. 
	If $F$ is of type II, then the edges of the corresponding subgraph of $\Gamma_{\tau}$ are $v_1v_2,v_2v_3,v_3v_1$.  
	For both types of $F$, $\Gamma_{\tau}$ contains the directed cycle $v_1v_2,v_2v_3,v_3v_1$ of length $3$. 
	
	$(2)$. By the statement (1), for every distinct vertices $v,w\in V$ that belong to the same face in $\Gamma$ there exists a walk from $v$ to $w$ in $\Gamma_{\tau}$. 
	Thus, by the connectivity of $\Gamma$, it follows that any two vertices of $\Gamma_{\tau}$ are connected by a walk (i.e. any two states of $\mathcal{X}_{\tau}$ communicate), which completes the proof of $(2)$. 
\end{proof}
Denote by $o(v)$ the {\it period} of a state $v\in V$, i.e. the greatest common divisor of the lengths of all closed walks in $\Gamma_{\tau}$ starting at $v$. 
Recall that if all states of the Markov chain have period $k$, then the Markov chain is called {\it periodic} with period $k$. 
If $o(v)=1$, then $v$ is {\it aperiodic}. 
If all states are aperiodic, then the Markov chain is also called {\it aperiodic}. 
All states of an irreducible Markov chain have the same period. 
\begin{prop}\label{prop2}
	The following statements are equivalent:
	\begin{enumerate}
		\item[$(1)$] $\mathcal{X}_{\tau}$ is aperiodic. 
		\item[$(2)$] $\mathcal{X}_{\tau}$ is ergodic. 
		\item[$(3)$] There exists a closed walk in $\Gamma_{\tau}$ whose length is not divisible by $3$. 
	\end{enumerate}
\end{prop}
\begin{proof}
	The equivalence of $(1)$ and $(2)$ follows from the second part of Proposition \ref{prop1}. 
	Now, we establish the equivalence of $(1)$ and $(3)$. 
	Assume that $\mathcal{X}_{\tau}$ is aperiodic and let $v\in V$. 
	By the statement (1) of Proposition \ref{prop1}, there exists a closed walk of length $3$ in $\Gamma_{\tau}$ passing through $v$. 
	Since $o(v)=1$, there exists another closed walk passing through $v$ whose length is not divisible by $3$. 
	Conversely, suppose that $\Gamma_{\tau}$ contains a closed walk whose length is not divisible by $3$. 
	If $v$ is a vertex of this walk, then the first statement of Proposition \ref{prop1} implies that $o(v)=1$ and the statement $(2)$ of the same proposition guarantees that $\mathcal{X}_{\tau}$ is aperiodic.
\end{proof}
Proposition \ref{prop2} implies that if $\mathcal{X}_{\tau}$ is not ergodic, then it is periodic with period $3$. 

\begin{prop}\label{lem1}
	If there exists a face of type I in $(\Gamma,\tau)$, then $\mathcal{X}_{\tau}$ is ergodic. 
\end{prop}
\begin{proof}
	Let $F$ be a face of type I in $\Gamma$ with the vertices $v_1, v_2, v_3$ and the edge of type II directed from $v_1$ to $v_2$ (as in the first part of the proof of Proposition \ref{prop1}). 
	Since $\Gamma_{\tau}$ contains the closed walk
	$$v_1v_2, v_2v_3, v_3v_2, v_2v_3, v_3v_1$$
	of length $5$, Proposition \ref{prop2} implies that $\mathcal{X}_{\tau}$ is ergodic. 
\end{proof}

%%%%%%%%%%%%%%%%%%%%%%%%%%%%%%%%%%%%%%%%%%%%%%%%%%%%%%%%%%%%%%%%%%%%%%%%%%%%%%%%%%%%%%%%%%%%%%%
\section{Main result} 
Our main result is the following. 
\begin{theorem}\label{th1} 
	Let $(\Gamma,\tau)$ be a $z$-oriented triangulation of a surface $M$. 
	The following assertions are fulfilled:
	\begin{enumerate}
		\item[(1)] The Markov chain $\mathcal{X}_{\tau}$ is not ergodic if and only if $(\Gamma,\tau)$ is $3$-colorable and all its faces are of type II. 
		\item[(2)] In the case when $M$ is the sphere or the real projective plane, $\mathcal{X}_{\tau}$ is not ergodic if and only if all faces of $(\Gamma,\tau)$ are of type II. 
	\end{enumerate}
\end{theorem}

\begin{exmp}\label{ex4}{\rm 
		The Markov chain of $\mathcal{O}$ from Example \ref{ex1} is ergodic for $\tau_1$ and $\tau_3$ since for these $z$-orientations $\mathcal{O}$ contains faces of type I. 
		In contrast, all faces of $(\mathcal{O},\tau_2)$ are of type II and the corresponding Markov chain is not ergodic. 
		The Markov chain of $\mathcal{T}_{k,m}$ (where $k,m\geq 3$) from Example \ref{ex2} is
		\begin{enumerate} 
			\item[$\bullet$] ergodic for $(\mathcal{T}_{k,m},\tau_1)$; 
			\item[$\bullet$] ergodic for $(\mathcal{T}_{k,m},\tau_2)$ if at least one of the numbers $k,m$ is not divisible by $3$ (i.e. when $\mathcal{T}_{k,m}$ is not $3$-colorable); 
			\item[$\bullet$] not ergodic for $(\mathcal{T}_{k,m},\tau_2)$ if both $k,m$ are divisible by $3$ (i.e. when $\mathcal{T}_{k,m}$ is $3$-colorable). 
		\end{enumerate} 
		For $\mathcal{P}$ from Example \ref{ex3} the corresponding Markov chain is 
		not ergodic for $\tau_1$ ($\mathcal{P}$ is $3$-colorable, see Fig. 5) and ergodic for $\tau_2$.  
}\end{exmp} 

It follows from the first part of Theorem \ref{th1} that precisely one of the following possibilities is realized for every $z$-oriented triangulation $(\Gamma,\tau)$ of $M$:
\begin{enumerate}
	\item[(A)] $(\Gamma,\tau)$ contains a face of type I and its Markov chain is ergodic.  
	\item[(B)] $(\Gamma,\tau)$ is $3$-colorable and all its faces are of type II; in this case, its  Markov chain is not ergodic.
	\item[(C)] $(\Gamma,\tau)$ is not $3$-colorable and all its faces are of type II which implies that its  Markov chain is ergodic.
\end{enumerate} 
Thus, by the second part of Theorem \ref{th1}, the case (C) does not occur when $M$ is the sphere. 
In Sections 6, we show that each of the above possibilities is realized for any orientable surface different from the sphere. 
In other words, the second part of Theorem \ref{th1} fails for each orientable surface different from the sphere. 
	
The second part of Theorem \ref{th1} shows that (C) is not realized for the real projective plane.  
Examples from Section 6 show that (A) and (B) are realized for every non-orientable surface. 
The realization of (C) on non-orientable surfaces different from the real projective plane is an open problem. 

%%%%%%%%%%%%%%%%%%%%%%%%%%%%%%%%%%%%%%%%%%%%%%%%%%%%%%%%%%%%%%%%%%%%%%%%%%%%%%%%%%%%%%%%%%%%%%%
\section{Proof of Theorem \ref{th1}} 
By \cite{TsaiWest}, an {\it Eulerian near-triangulation} is a multigraph embedded in the sphere in which every vertex has even degree and at most one face has the boundary that is not a $3$-cycle. 
%An {\it Eulerian near-triangulation} is a planar multigraph whose bounded faces are all $3$-cycles and every vertex has even degree. 
To prove Theorem \ref{th1} we need the following.
\begin{theorem}[\cite{TsaiWest}]\label{ET3}
	Every Eulerian near-triangulation is $3$-colorable. 
\end{theorem}
Note that the underlying graph of a directed Eulerian multi-triangulation of the sphere is an Eulerian near-triangulation; hence, this graph is $3$-colorable. 
In particular, every Eulerian triangulation of the sphere is $3$-colorable. 
\begin{lemma}\label{lem2}
	Let $\Gamma_0$ be a directed Eulerian multi-triangulation of a surface $M$ whose underlying graph is $3$-colorable. 
	The following statements hold:
	\begin{enumerate}
		\item[$(1)$] Suppose that there exists a walk of length $m$ from a vertex $v$ to a vertex $w$. 
		Then $v$ and $w$ have the same color if and only if $m$ is divisible by $3$. 
		%If for vertices $v,w$ there exists a walk of length $m$ from $v$ to $w$, then $v$ and $w$ have the same color if and only if $m$ is divisible by $3$. 
		\item[$(2)$] Every closed walk has length divisible by $3$. 
	\end{enumerate}
\end{lemma}
\begin{proof}
	$(1)$. Let $v_0$ be a vertex of $\Gamma_0$ and let $\mathrm{outdeg}(v_0)=\mathrm{indeg}(v_0)=k$. 
	Assume that (with indices taken modulo $2k$)
	$$w_0,w_1,\dots,w_{2k-2},w_{2k-1}$$ 
	are the vertices adjacent to $v_0$ cyclically ordered according to one of two local orientations around $v_0$ (clockwise or anticlockwise). 
	Since the underlying graph of $\Gamma_0$ need not be simple, it is possible that $w_i=w_j$ for distinct $i,j\in\{0,\dots,2k-1\}$. 
	%It is also possible that for such $i,j$, the distinct faces with vertices are $v_0,w_i,w_{i+1}$ and $v_0,w_j,w_{j+1}$ (respectively) have a common directed edge whose vertices are $v_i=v_j$ and $w_{i+1}=w_{j+1}$. 
	For every $i\in\{0,\dots,k-1\}$ the vertices $v_0,w_{2i},w_{2i+1}$ belong to a face. 
	Thus, without loss of generality suppose that they form the following directed edges
	$$v_0w_{2i},\quad w_{2i}w_{2i+1},\quad w_{2i+1}v_0.$$ 
	Hence, the face whose vertices are $w_{2i},w_{2i-1},v_0$ has the directed edge $w_{2i}w_{2i-1}$, see Fig. 8. 
		\begin{center}
	\begin{tikzpicture}[scale=0.5]
		
		\draw[fill=black] (0,0) circle (3pt);
		
		\draw[fill=black] (0:4cm) circle (3pt);
		\draw[fill=black] (60:4cm) circle (3pt);
		\draw[fill=black] (120:4cm) circle (3pt);		
		\draw[fill=black] (180:4cm) circle (3pt);
		\draw[fill=black] (240:4cm) circle (3pt);
		\draw[fill=black] (300:4cm) circle (3pt);		
		
		%center
		\draw [thick, line width=1pt, decoration={markings,
			mark=at position 0.55 with {\arrow[scale=1.25,>=stealth]{<}}},
		postaction={decorate}] (0:4cm) -- (0,0);
		
		\draw [thick, line width=1pt, decoration={markings,
			mark=at position 0.55 with {\arrow[scale=1.25,>=stealth]{<}}},
		postaction={decorate}] (0,0) -- (60:4cm);
		
		\draw [thick, line width=1pt, decoration={markings,
			mark=at position 0.55 with {\arrow[scale=1.25,>=stealth]{<}}},
		postaction={decorate}] (120:4cm) -- (0,0);
		
		\draw [thick, line width=1pt, decoration={markings,
			mark=at position 0.55 with {\arrow[scale=1.25,>=stealth]{<}}},
		postaction={decorate}] (0,0) -- (180:4cm);
		
		\draw [thick, line width=1pt, decoration={markings,
			mark=at position 0.55 with {\arrow[scale=1.25,>=stealth]{<}}},
		postaction={decorate}] (240:4cm) -- (0,0);
		
		\draw [thick, line width=1pt, decoration={markings,
			mark=at position 0.55 with {\arrow[scale=1.25,>=stealth]{<}}},
		postaction={decorate}] (0,0) -- (300:4cm);
		
		%outside
		\draw [thick, line width=1pt, decoration={markings,
			mark=at position 0.55 with {\arrow[scale=1.25,>=stealth]{<}}},
		postaction={decorate}] (60:4cm) -- (0:4cm);
		
		\draw [thick, line width=1pt, decoration={markings,
			mark=at position 0.56 with {\arrow[scale=1.25,>=stealth]{<}}},
		postaction={decorate}] (300:4cm) -- (0:4cm);
		
		\draw [thick, line width=1pt, decoration={markings,
			mark=at position 0.55 with {\arrow[scale=1.25,>=stealth]{<}}},
		postaction={decorate}] (300:4cm) -- (240:4cm);
		
		\draw [thick, line width=1pt, decoration={markings,
			mark=at position 0.56 with {\arrow[scale=1.25,>=stealth]{<}}},
		postaction={decorate}] (180:4cm) -- (240:4cm);
		
		\draw [thick, line width=1pt, decoration={markings,
			mark=at position 0.55 with {\arrow[scale=1.25,>=stealth]{<}}},
		postaction={decorate}] (180:4cm) -- (120:4cm);
		
		\draw [thick, line width=1pt] (2cm,3.4641cm) -- (1.25cm,3.4641cm);
		\draw [thick, line width=1pt] (-2cm,3.4641cm) -- (-1.25cm,3.4641cm);
		
		\draw[line width=1.5pt, dash pattern=on 0pt off 4pt, dotted] (-1.25cm,3.4641cm) -- (1.25cm,3.4641cm);
		
		\draw [thick, line width=1pt, dashed] (0,0) -- (72:1cm);
		\draw [thick, line width=1pt, dashed] (0,0) -- (84:1cm);
		\draw [thick, line width=1pt, dashed] (0,0) -- (96:1cm);
		\draw [thick, line width=1pt, dashed] (0,0) -- (108:1cm);
		
		\node at (0.57,-0.35) {$v_0$};
		\node at (248:3.45cm) {$c_1$};
		\node at (306:3.42cm) {$c_3$};
		\node at (6:3.42cm) {$c_1$};
		\node at (66:3.42cm) {$c_3$};
		\node at (125:3.42cm) {$c_1$};
		\node at (186:3.42cm) {$c_3$};
		
		\node at (0cm,-0.7cm) {$c_2$};
		
		\node at (0:4.6cm) {$w_2$};
		\node at (60:4.48cm) {$w_3$};
		\node at (120:4.48cm) {$w_{2k-2}$};
		\node at (180:5.1cm) {$w_{2k-1}$};
		\node at (240:4.48cm) {$w_0$};
		\node at (300:4.48cm) {$w_1$};
		
		\node at (0:6cm) {\color{white} .};
		
		\end{tikzpicture}
		\captionof{figure}{ }
	\end{center}
	Let $c_1,c_2,c_3$ be colors assigned to the vertices of $\Gamma_0$ in its underlying graph $3$-coloring with $w_0$ and $v_0$ colored in $c_1$ and $c_2$, respectively. 
	Any two consecutive vertices in the cyclic sequence $w_0,w_1,\dots,w_{2k-2},w_{2k-1}$ form a face with $v_0$ and thus the vertices of this sequence must be colored in $c_1$ and $c_3$ alternately. 
	%More precisely, every edge directed out of $v_0$ is of form $v_0w_{2i}$ with $w_{2i}$ colored in $c_1$ and every edge directed into $v_0$  is of form $w_{2i+1}v_0$ with $w_{2i+1}$ colored in $c_3$ (for some $i\in\{0,\dots,k-1\}$), see Fig. 8. 
	Therefore, every walk in $\Gamma_0$ passing through $v_0$ contains a subwalk of the form $w_{2i+1},v_0,w_{2j}$ for some $i,j\in\{0,\dots,k-1\}$ and the vertices of this subwalk are colored in $c_3,c_2,c_1$, respectively. 
	Since the choice of $v_0$ is arbitrary, any $3$ consecutive vertices in any walk have mutually distinct colors. 
	Now, consider a walk of length $m$ connecting vertices $v,w$. 
	We write this walk as the sequence of vertices 
	$$v=u_0,u_1,u_2\dots,u_{m-1},u_m=w.$$ 
	Then, any $3$ consecutive vertices $u_i, u_{i+1}, u_{i+2}$ have mutually distinct colors and for every $i\in\{0,\dots,m-3\}$ the vertex $u_{i+3}$ has the same color as $u_i$. 
	Hence, $v$ and $w$ have the same color if and only if $m$ is divisible by $3$. 
	
	$(2)$. It follows from the statement $(1)$ for $v=w$. 
\end{proof}

%%%Proof of TH1
\begin{proof}[Proof of the statement {\rm (1)} of Theorem \ref{th1}]
Let $(\Gamma,\tau)$ be a $z$-oriented triangulation of $M$. 
Suppose that $(\Gamma,\tau)$ is $3$-colorable and all its faces are of type II. 
Therefore, $\Gamma_{\tau}$ is a directed Eulerian triangulation of $M$ and, by the second statement of Lemma \ref{lem2}, every closed walk in $\Gamma_{\tau}$ has length divisible by $3$. 
Using Proposition \ref{prop2} we conclude that $\mathcal{X}_{\tau}$ is not ergodic. 

Suppose now that $\mathcal{X}_{\tau}$ is not ergodic. 
Then, by Proposition \ref{lem1}, all faces of $(\Gamma, \tau)$ are of type II and $\Gamma_{\tau}$ is a directed Eulerian triangulation of $M$. 
By Proposition \ref{prop2}, every closed walk in $\Gamma_{\tau}$ has length divisible by $3$. 

%%%%%%%%%
%%%%%%%%%%construction
Consider the dual graph $\Gamma^{*}$ of $\Gamma$ and any spanning tree $T$ of $\Gamma^{*}$. 
The graphs $\Gamma$ and $\Gamma^*$ have the same set of edges and we denote by $U$ the subgraph of $\Gamma$ consisting of all edges which are not edges of $T$. 
%Let $U$ be the subgraph of $\Gamma$ whose edges are corresponding to edges of $\Gamma^{*}$ that do not belong to $T$. 
Since $T$ is a spanning tree of $\Gamma^{*}$, all vertices of $\Gamma$ belong to $U$. 
We remove $U$ from $M$ and obtain $M\setminus U$, which is homeomorphic to an open disk. 
There exists a unique fundamental polygon $D$ of $M$ whose interior is $M\setminus U$. 
Note that each edge of $U$ corresponds to precisely $2$ sides of $D$, but a vertex of $U$ may correspond to several vertices of $D$. 
Thus, we also obtain from $\Gamma$ a graph embedded in $D$ whose vertex set consists of the vertices of $D$ and whose edge set is the union of the set of sides of $D$ and the set of all edges of $\Gamma$ that do not belong to $U$. 
For the reason that $\Gamma$ is the underlying graph of $\Gamma_{\tau}$, we consider this new graph as a digraph $\Gamma_{D}$ with edge directions inherited from $\Gamma_{\tau}$. 
However, $\Gamma_{D}$ need not be Eulerian (see Fig. 9 for an example of such a construction for the transition graph of $(\mathcal{O},\tau_2)$ from Example \ref{ex1}).
\begin{center} 
	\begin{tikzpicture}[scale=0.48] 
		%\draw[step=0.5, gray, thin] (-15,0) grid (15,8);
		%\draw[step=1, red, thin] (-15,0) grid (15,8);
		%\draw[step=5, blue, thin] (-15,0) grid (15,8);	
		\begin{scope}
			\coordinate (A1) at (90:5cm);
			\coordinate (A2) at (210:5cm);
			\coordinate (A3) at (330:5cm);
			
			\coordinate (B1) at (270:1.25cm);
			\coordinate (B2) at (30:1.25cm);
			\coordinate (B3) at (150:1.25cm);
			
			%%%%%%%%%%%% 
			\draw[thick, line width=1.1pt, decoration={markings,
				mark=at position 0.5 with {\arrow[scale=1,>=stealth]{>}}},
			postaction={decorate}] (A1) -- (A2);%
			\draw[thick, line width=0.5pt, decoration={markings,
				mark=at position 0.5 with {\arrow[scale=1.3,>=stealth]{>}}},
			postaction={decorate}] (A2) -- (A3);
			\draw[thick, line width=1.1pt, decoration={markings,
				mark=at position 0.5 with {\arrow[scale=1,>=stealth]{>}}},
			postaction={decorate}] (A3) -- (A1);%
			
			\draw[thick, line width=0.5pt, decoration={markings,
				mark=at position 0.6 with {\arrow[scale=1.3,>=stealth]{>}}},
			postaction={decorate}] (B1) -- (B2);
			\draw[thick, line width=0.5pt, decoration={markings,
				mark=at position 0.6 with {\arrow[scale=1.3,>=stealth]{>}}},
			postaction={decorate}] (B2) -- (B3);
			\draw[thick, line width=0.5pt, decoration={markings,
				mark=at position 0.6 with {\arrow[scale=1.3,>=stealth]{>}}},
			postaction={decorate}] (B3) -- (B1);
			
			\draw[thick, line width=0.5pt, decoration={markings,
				mark=at position 0.6 with {\arrow[scale=1.3,>=stealth]{>}}},
			postaction={decorate}] (A1) -- (B2);
			\draw[thick, line width=1.1pt, decoration={markings,
				mark=at position 0.6 with {\arrow[scale=1,>=stealth]{<}}},
			postaction={decorate}] (A1) -- (B3);%
			
			\draw[thick, line width=1.1pt, decoration={markings,
				mark=at position 0.6 with {\arrow[scale=1,>=stealth]{<}}},
			postaction={decorate}] (A2) -- (B1);%
			\draw[thick, line width=0.5pt, decoration={markings,
				mark=at position 0.6 with {\arrow[scale=1.3,>=stealth]{>}}},
			postaction={decorate}] (A2) -- (B3);
			
			\draw[thick, line width=0.5pt, decoration={markings,
				mark=at position 0.6 with {\arrow[scale=1.3,>=stealth]{>}}},
			postaction={decorate}] (A3) -- (B1);
			\draw[thick, line width=1.1pt, decoration={markings,
				mark=at position 0.6 with {\arrow[scale=1,>=stealth]{<}}},
			postaction={decorate}] (A3) -- (B2);%
			%%%%%%%%%%%%%%%%%%%5
			
			\node at (-2cm,3cm) {$\boldsymbol{U}$};

			\node at (90:5.5cm) {$a_1$};
			\node at (215:5.5cm) {$a_2$};
			\node at (325:5.5cm) {$a_3$};
			
			\node at (27:1.8cm) {$b_2$};
			\node at (183:1.1cm) {$b_3$};
			\node at (307:1.25cm) {$b_1$};	
			
			\node at (-5cm,0cm) {$\Gamma_{\tau}$};
			
			\draw[fill=black] (A1) circle (3.5pt);
			\draw[fill=black] (A2) circle (3.5pt);	
			\draw[fill=black] (A3) circle (3.5pt);
			
			\draw[fill=black] (B1) circle (3.5pt);
			\draw[fill=black] (B2) circle (3.5pt);	
			\draw[fill=black] (B3) circle (3.5pt);
			
			%%%%%%%%%%%%%%tree
			\coordinate (T0) at ($ (0,0)!0.333333!($(B1)+(B2)+(B3)$) $);
			
			\coordinate (T1) at ($ (0,0)!0.333333!($(B1)+(A2)+(B3)$) $);
			\coordinate (T2) at ($ (0,0)!0.333333!($(B1)+(B2)+(A3)$) $);
			\coordinate (T3) at ($ (0,0)!0.333333!($(A1)+(B2)+(B3)$) $);
			
			\coordinate (T4) at ($ (0,0)!0.333333!($(A1)+(A2)+(B3)$) $);
			\coordinate (T5) at ($ (0,0)!0.333333!($(A1)+(B2)+(A3)$) $);
			\coordinate (T6) at ($ (0,0)!0.333333!($(B1)+(A2)+(A3)$) $);
			
			\coordinate (T7) at (0,-3.5cm);

			\draw[dashed, thick, line width=0.5pt] (T0) -- (T1);
			\draw[dashed, thick, line width=0.5pt] (T4) -- (T1);
			\draw[dashed, thick, line width=0.5pt] (T0) -- (T2);
			\draw[dashed, thick, line width=0.5pt] (T6) -- (T2);
			\draw[dashed, thick, line width=0.5pt] (T0) -- (T3);
			\draw[dashed, thick, line width=0.5pt] (T5) -- (T3);
			
			\draw[dashed, thick, line width=0.5pt] (T6) -- (T7);
			
			\draw[fill=white] (T0) circle (3.5pt);
			\draw[fill=white] (T1) circle (3.5pt);
			\draw[fill=white] (T2) circle (3.5pt);
			\draw[fill=white] (T3) circle (3.5pt);
			\draw[fill=white] (T4) circle (3.5pt);
			\draw[fill=white] (T5) circle (3.5pt);
			\draw[fill=white] (T6) circle (3.5pt);
			\draw[fill=white] (T7) circle (3.5pt);
			
			\node at (0.6cm,-3.5cm) {$T$};
			\end{scope}
			%%polygon
			\begin{scope}[xshift=12cm,yshift=1cm, rotate=180,xscale=-1]
				\coordinate (w1) at (0:4cm);
				\coordinate (w2) at (36:4cm);
				\coordinate (w3) at (72:4cm);
				\coordinate (w4) at (108:4cm);
				\coordinate (w5) at (144:4cm);
				\coordinate (w6) at (180:4cm);
				\coordinate (w7) at (216:4cm);
				\coordinate (w8) at (252:4cm);
				\coordinate (w9) at (288:4cm);
				\coordinate (w10) at (324:4cm);
				
				\draw[thick, line width=1.1pt, decoration={markings,
					mark=at position 0.55 with {\arrow[scale=1,>=stealth]{<}}},
				postaction={decorate}] (w1) -- (w2); 
				
				\draw[thick, line width=1.1pt, decoration={markings,
					mark=at position 0.55 with {\arrow[scale=1,>=stealth]{>}}},
				postaction={decorate}] (w2) -- (w3);
				
				\draw[thick, line width=1.1pt, decoration={markings,
					mark=at position 0.55 with {\arrow[scale=1,>=stealth]{>}}},
				postaction={decorate}] (w3) -- (w4);
				
				\draw[thick, line width=1.1pt, decoration={markings,
					mark=at position 0.55 with {\arrow[scale=1,>=stealth]{>}}},
				postaction={decorate}] (w4) -- (w5);
				
				\draw[thick, line width=1.1pt, decoration={markings,
					mark=at position 0.55 with {\arrow[scale=1,>=stealth]{<}}},
				postaction={decorate}] (w5) -- (w6);
				
				\draw[thick, line width=1.1pt, decoration={markings,
					mark=at position 0.55 with {\arrow[scale=1,>=stealth]{>}}},
				postaction={decorate}] (w6) -- (w7);
				
				\draw[thick, line width=1.1pt, decoration={markings,
					mark=at position 0.55 with {\arrow[scale=1,>=stealth]{<}}},
				postaction={decorate}] (w7) -- (w8);
				
				\draw[thick, line width=1.1pt, decoration={markings,
					mark=at position 0.55 with {\arrow[scale=1,>=stealth]{<}}},
				postaction={decorate}] (w8) -- (w9); 
				
				\draw[thick, line width=1.1pt, decoration={markings,
					mark=at position 0.55 with {\arrow[scale=1,>=stealth]{>}}},
				postaction={decorate}] (w9) -- (w10); 
				
				\draw[thick, line width=1.1pt, decoration={markings,
					mark=at position 0.55 with {\arrow[scale=1,>=stealth]{<}}},
				postaction={decorate}] (w10) -- (w1);
				
				%interior
				\draw[thick, line width=0.5pt, decoration={markings,
					mark=at position 0.56 with {\arrow[scale=1.3,>=stealth]{>}}},
				postaction={decorate}] (w2) -- (w9);
				
				\draw[thick, line width=0.5pt, decoration={markings,
					mark=at position 0.56 with {\arrow[scale=1.3,>=stealth]{>}}},
				postaction={decorate}] (w3) -- (w6);
				
				\draw[thick, line width=0.5pt, decoration={markings,
					mark=at position 0.56 with {\arrow[scale=1.3,>=stealth]{>}}},
				postaction={decorate}] (w5) -- (w3);
				
				\draw[thick, line width=0.5pt, decoration={markings,
					mark=at position 0.56 with {\arrow[scale=1.3,>=stealth]{>}}},
				postaction={decorate}] (w6) -- (w2);
				
				\draw[thick, line width=0.5pt, decoration={markings,
					mark=at position 0.56 with {\arrow[scale=1.3,>=stealth]{>}}},
				postaction={decorate}] (w7) -- (w9);
				
				\draw[thick, line width=0.5pt, decoration={markings,
					mark=at position 0.56 with {\arrow[scale=1.3,>=stealth]{>}}},
				postaction={decorate}] (w9) -- (w6);
				
				\draw[thick, line width=0.5pt, decoration={markings,
					mark=at position 0.56 with {\arrow[scale=1.3,>=stealth]{>}}},
				postaction={decorate}] (w10) -- (w2);

				\draw[fill=black] (w1) circle (3.5pt);
				\draw[fill=black] (w2) circle (3.5pt);
				\draw[fill=black] (w3) circle (3.5pt);
				\draw[fill=black] (w4) circle (3.5pt);
				\draw[fill=black] (w5) circle (3.5pt);
				\draw[fill=black] (w6) circle (3.5pt);
				\draw[fill=black] (w7) circle (3.5pt);
				\draw[fill=black] (w8) circle (3.5pt);
				\draw[fill=black] (w9) circle (3.5pt);
				\draw[fill=black] (w10) circle (3.5pt);
				
				%tree
				\coordinate (S0) at ($ (0,0)!0.333333!($(w2)+(w6)+(w9)$) $);
				
				\coordinate (S1) at ($ (0,0)!0.333333!($(w6)+(w7)+(w9)$) $);
				\coordinate (S2) at ($ (0,0)!0.333333!($(w2)+(w3)+(w6)$) $);
				\coordinate (S3) at ($ (0,0)!0.333333!($(w2)+(w9)+(w10)$) $);
				
				\coordinate (S4) at ($ (0,0)!0.333333!($(w7)+(w8)+(w9)$) $);
				\coordinate (S5) at ($ (0,0)!0.333333!($(w1)+(w2)+(w10)$) $);
				\coordinate (S6) at ($ (0,0)!0.333333!($(w3)+(w5)+(w6)$) $);
				
				\coordinate (S7) at ($ (0,0)!0.333333!($(w3)+(w4)+(w5)$) $);

				\draw[dashed, thick, line width=0.5pt] (S0) -- (S1);
				\draw[dashed, thick, line width=0.5pt] (S4) -- (S1);
				\draw[dashed, thick, line width=0.5pt] (S0) -- (S2);
				\draw[dashed, thick, line width=0.5pt] (S6) -- (S2);
				\draw[dashed, thick, line width=0.5pt] (S0) -- (S3);
				\draw[dashed, thick, line width=0.5pt] (S5) -- (S3);
				
				\draw[dashed, thick, line width=0.5pt] (S6) -- (S7);
				
				\draw[fill=white] (S0) circle (3.5pt);
				\draw[fill=white] (S1) circle (3.5pt);
				\draw[fill=white] (S2) circle (3.5pt);
				\draw[fill=white] (S3) circle (3.5pt);
				\draw[fill=white] (S4) circle (3.5pt);
				\draw[fill=white] (S5) circle (3.5pt);
				\draw[fill=white] (S6) circle (3.5pt);
				\draw[fill=white] (S7) circle (3.5pt);
				
				\node at ($(S0)+(-0.5cm,0.5cm)$) {$T$};
				%%%
				
				\node at (0:4.75cm) {$a_3$};
				\node at (36:4.75cm) {$b_2$};
				\node at (72:4.75cm) {$a_3$};
				\node at (108:4.75cm) {$a_1$};
				\node at (144:4.75cm) {$a_2$};
				\node at (180:4.75cm) {$b_1$};
				\node at (216:4.75cm) {$a_2$};
				\node at (252:4.75cm) {$a_1$};
				\node at (288:4.75cm) {$b_3$};
				\node at (324:4.75cm) {$a_1$};
				
				\node at (6.5cm,0) {$\Gamma_{D}$};
				\node at (308:4.8cm) {$\boldsymbol{\partial D}$};
			\end{scope}
			
	\end{tikzpicture}
	\captionof{figure}{ }
\end{center}
%%%%%%%%%%
\iffalse
Recall that any closed surface can be cut along a suitable subset of edges of its triangulation to obtain a fundamental polygon. 
We apply this operation to $\Gamma_{\tau}$ and get a fundamental polygon $D$ of $M$ 
together with a digraph $\Gamma_{D}$ embedded in $D$ whose edge directions are inherited from $\Gamma_{\tau}$. 
The directed edges of $\Gamma_{D}$ forming $\partial D$ are paired and each pair corresponds to a single edge of $\Gamma_{\tau}$. 
Note that the underlying graph of $\Gamma_{D}$ need not be Eulerian. 
\fi
%%%%%%%%%%%
\iffalse%
Let a $2n$-gon $D$ be a fundamental polygon of $M$ ($n\geq 1$). 
Since the sides of $D$ are matched in pairs such that the sides from one pair are identified with the same curve in $M$, 
we assume that each such curve is formed by a sequence of edges of $\Gamma$ (these edges do not necessarily form a directed walk in $\Gamma_{\tau}$). 
By cutting the embedding of $\Gamma_{\tau}$ in $M$ along these sequences of edges, we obtain a new digraph $\Gamma_{D}$ embedded in $D$. 
Observe that every directed edge of $\Gamma_{\tau}$ that is cut in this construction corresponds to two directed edges of $\Gamma_{D}$ contained in the boundary $\partial D$ of the polygon $D$. 
\fi%

We exploit $\Gamma_{D}$ in constructing the digraph $\Gamma_0$ on the sphere.
Consider the embedding of a digraph $\Gamma_{D'}$ in the polygon $D'$, which is a copy of the embedding of $\Gamma_{D}$ in $D$. 
Then, there exists an isomorphism
$i:\Gamma_{D}\to\Gamma_{D'}$ 
such that
$i(v)=v',$ 
where $v'$ is a vertex of $\Gamma_{D'}$ obtained as a copy of a vertex $v$ of $\Gamma_{D}$. 
We identify the boundaries $\partial D$ and $\partial D'$ according to this isomorphism and we obtain a digraph $\Gamma_0$. 
It is easy to see that for each vertex of $\Gamma_0$ the outdegree equals the indegree  
and all faces of $\Gamma_0$ are triangles. 
Therefore, $\Gamma_0$ is a directed Eulerian multi-triangulation of the sphere and its underlying graph is $3$-colorable. 

To prove that $\Gamma$ %(the underlying graph of $\Gamma_{\tau}$) 
is $3$-colorable, it suffices to show that if two vertices of $\Gamma_{D}$ are obtained from the same vertex of $\Gamma_{\tau}$, then these vertices have the same color in the $3$-coloring of the underlying graph of $\Gamma_0$. 

Let $v,w$ be vertices of $\Gamma_0$ obtained from the same vertex of $\Gamma_{\tau}$. 
%There exists a walk from $v$ to $w$ in $\Gamma_0$. 
%Indeed, such a walk can be constructed in the following way. 
We take any walk 
in the underlying graph of $\Gamma_0$. 
This walk is not necessarily a directed walk in $\Gamma_0$. 
So, if this walk contains subsequence $v',w'$ such that $w'v'$ is a directed edge of $\Gamma_0$, then there exists a vertex $u$ such that $w'v',v'u,uw'$ are directed edges of one of faces of $\Gamma_0$ and we insert $u$ between $v'$ and $w'$ in the walk.  In other words, we replace $v'w'$ with two consecutive directed edges $v'u,uw'$. 
In this way, we construct a walk from $v$ to $w$ in $\Gamma_0$. 

Note that if there is a directed edge in $\Gamma_{D'}$ (as a subgraph of $\Gamma_0$) that is not contained in $\partial D=\partial D'$, then there exists an edge in $\Gamma_{D}$  with the same vertices and the same direction.
Thus, there exists a walk $P$ from $v$ to $w$ in the subgraph $\Gamma_{D}$ of $\Gamma_0$. 
Since $P$ corresponds to a closed walk of the same length in $\Gamma_{\tau}$, 
the length of $P$ is divisible by $3$. 
Therefore, by the first part of Lemma \ref{lem2}, $v$ and $w$ have the same color.
\end{proof}

%Proof for the sphere and RP2
Recall that for a graph $G$ and a subset $U$ of its vertex set, the graph $G - U$ is obtained from $G$ by deleting all vertices of $U$ and all edges incident to them. 
If $G$ is embedded in a surface, then a {\it color factor} of $G$ is a subset $\mathcal{C}$ of its vertex set such that every face of $G$ has precisely one vertex in $\mathcal{C}$. 

To prove the second part of Theorem \ref{th1} we need the following result (see \cite{Fisk,MoharProj}). 
\begin{theorem}\label{THmohar}
	Let $G$ be an Eulerian triangulation of the real projective plane. 
	Then $G$ has a color factor. 
	If $\mathcal{C}$ is any color factor of $G$, then $G$ is $3$-colorable if and only if $G-\mathcal{C}$ is bipartite. 
\end{theorem}
\begin{proof}[Proof of the statement {\rm (2)} of Theorem \ref{th1}]
If $M$ is the sphere, then it follows from the first part of Theorem \ref{th1} and Theorem \ref{ET3}. 
Assume that $M$ is the real projective plane. 
Suppose that all faces of $(\Gamma, \tau)$ are of type II. 
Therefore, $\Gamma$ is an Eulerian triangulation of $M$ containing a color factor $\mathcal{C}\subset V$. 
Given that all edges of $(\Gamma, \tau)$ are of type II, we consider $\Gamma - \mathcal{C}$ as a directed graph with edge directions
induced by $\tau$. 
Let $v$ be a vertex of $\Gamma - \mathcal{C}$ and let 
$$v_0,v_0',v_1,v'_1,\dots,v_{m-1},v'_{m-1},$$
where $m\geq 2$, be the vertices of $\Gamma$ incident to $v$, cyclically ordered according to one of two local orientations around $v$ (clockwise or anticlockwise). 
Without loss of generality, we suppose that for each $i\in\{0,\dots,m-1\}$ (with indices taken modulo $m$):
\begin{enumerate}
	\item[$\bullet$] the face with vertices $v,v_i,v'_i$ has the following edges of type II: $v_iv,vv'_i,v'_iv_i$;
	\item[$\bullet$] the face with vertices $v,v'_i,v_{i+1}$ has the following edges of type II: $vv'_i,v'_iv_{i+1},v_{i+1}v$,
\end{enumerate}
see Fig. 10.
		\begin{center}
	\begin{tikzpicture}[scale=0.5]
		
		\draw[fill=black] (0,0) circle (3pt);
		
		\draw[fill=black] (0:4cm) circle (3pt);
		\draw[fill=black] (60:4cm) circle (3pt);
		\draw[fill=black] (120:4cm) circle (3pt);		
		\draw[fill=black] (180:4cm) circle (3pt);
		\draw[fill=black] (240:4cm) circle (3pt);
		\draw[fill=black] (300:4cm) circle (3pt);	
		
		%center
		\draw [thick, line width=1pt, decoration={markings,
			mark=at position 0.6 with {\arrow[scale=1.25,>=stealth]{<<}}},
		postaction={decorate}] (0:4cm) -- (0,0);
		
		\draw [thick, line width=1pt, decoration={markings,
			mark=at position 0.6 with {\arrow[scale=1.25,>=stealth]{<<}}},
		postaction={decorate}] (0,0) -- (60:4cm);
		
		\draw [thick, line width=1pt, decoration={markings,
			mark=at position 0.6 with {\arrow[scale=1.25,>=stealth]{<<}}},
		postaction={decorate}] (120:4cm) -- (0,0);
		
		\draw [thick, line width=1pt, decoration={markings,
			mark=at position 0.6 with {\arrow[scale=1.25,>=stealth]{<<}}},
		postaction={decorate}] (0,0) -- (180:4cm);
		
		\draw [thick, line width=1pt, decoration={markings,
			mark=at position 0.6 with {\arrow[scale=1.25,>=stealth]{<<}}},
		postaction={decorate}] (240:4cm) -- (0,0);
		
		\draw [thick, line width=1pt, decoration={markings,
			mark=at position 0.6 with {\arrow[scale=1.25,>=stealth]{<<}}},
		postaction={decorate}] (0,0) -- (300:4cm);
		
		%outside
		\draw [thick, line width=1pt, decoration={markings,
			mark=at position 0.6 with {\arrow[scale=1.25,>=stealth]{<<}}},
		postaction={decorate}] (60:4cm) -- (0:4cm);
		
		\draw [thick, line width=1pt, decoration={markings,
			mark=at position 0.6 with {\arrow[scale=1.25,>=stealth]{<<}}},
		postaction={decorate}] (300:4cm) -- (0:4cm);
		
		\draw [thick, line width=1pt, decoration={markings,
			mark=at position 0.6 with {\arrow[scale=1.25,>=stealth]{<<}}},
		postaction={decorate}] (300:4cm) -- (240:4cm);
		
		\draw [thick, line width=1pt, decoration={markings,
			mark=at position 0.6 with {\arrow[scale=1.25,>=stealth]{<<}}},
		postaction={decorate}] (180:4cm) -- (240:4cm);
		
		\draw [thick, line width=1pt, decoration={markings,
			mark=at position 0.6 with {\arrow[scale=1.25,>=stealth]{<<}}},
		postaction={decorate}] (180:4cm) -- (120:4cm);
		
		\draw [thick, line width=1pt] (2cm,3.4641cm) -- (1.25cm,3.4641cm);
		\draw [thick, line width=1pt] (-2cm,3.4641cm) -- (-1.25cm,3.4641cm);
		
		\draw[line width=1.5pt, dash pattern=on 0pt off 4pt, dotted] (-1.25cm,3.4641cm) -- (1.25cm,3.4641cm);
		
		\draw [thick, line width=1pt, dashed] (0,0) -- (72:1cm);
		\draw [thick, line width=1pt, dashed] (0,0) -- (84:1cm);
		\draw [thick, line width=1pt, dashed] (0,0) -- (96:1cm);
		\draw [thick, line width=1pt, dashed] (0,0) -- (108:1cm);
		
		\node at (0.55,-0.38) {$v$};
		
		\node at (0:4.95cm) {$v'_{m-1}$};
		\node at (60:4.55cm) {$v_{m-1}$};
		\node at (120:4.55cm) {$v'_1$};
		\node at (180:4.5cm) {$v_1$};
		\node at (240:4.55cm) {$v'_0$};
		\node at (300:4.55cm) {$v_0$};
		
		\node[color=white] at (180:5.25cm) {$.$};
		%\draw[step=0.5cm, gray, very thin] (-4.5,-6) grid (5,5);
		
	\end{tikzpicture}
	\captionof{figure}{ }
\end{center}
Since $v\notin\mathcal{C}$, either
$$\{v_0,\dots,v_{m-1}\}\subset\mathcal{C}\quad\text{and}\quad\{v'_0,\dots,v'_{m-1}\}\cap\mathcal{C}=\varnothing$$
or
$$\{v'_0,\dots,v'_{m-1}\}\subset\mathcal{C}\quad\text{and}\quad\{v_0,\dots,v_{m-1}\}\cap\mathcal{C}=\varnothing.$$
In other words, all edges of $\Gamma - \mathcal{C}$ incident to $v$ are directed in the same way: 
either all out of $v$ or all into $v$. 
Thus, $\Gamma - \mathcal{C}$ is bipartite and (by Theorem \ref{THmohar}) $\Gamma$ is $3$-colorable. 
Hence, the $3$-colorability condition from the first part of Theorem \ref{th1} can be omitted in the case of the real projective plane. 
\end{proof}

%%%%%%%%%%%%%%%%%%%%%%%%%%%%%%%%%%%%%%%%%%%%%%%%%%%%%%%%%%%%%%%%%%%%%%%%%%%%%%%%%%%%%%%%%%%%%%%
\section{Examples}
%connected sums
Let $\Gamma$ and $\Gamma'$ be triangulations of connected closed $2$-dimensional surfaces $M$ and $M'$ and let $F$ and $F'$ be their faces (respectively). 
A homeomorphism $g:\partial F\to\partial F'$ is called {\it special} if it transfers each vertex to a vertex. 
The triangulation $\Gamma\#_g\Gamma'$ of $M\# M'$ is obtained from $\Gamma$ and $\Gamma'$ by removing the interiors of $F$ and $F'$ (respectively) and gluing their boundaries $\partial F$ and $\partial F'$ by $g$. 
We say that $\Gamma\#_g\Gamma'$ is the {\it connected sum} of $\Gamma$ and $\Gamma'$. 

\begin{lemma}\label{lem3}
	Let $(\Gamma,\tau)$ and $(\Gamma',\tau')$ be $z$-oriented triangulations with all faces of type II. 
	For any faces $F,F'$ from $\Gamma,\Gamma'$ (respectively) and for every special homeomorphism $g:\partial F\to\partial F'$, there exists a $z$-orientation of $\Gamma\#_g\Gamma'$ such that all its faces are of type II. 
\end{lemma}
\begin{proof}
	Denote the vertices of $F$ by $v_0,v_1,v_2$. 
	For $\tau$, the edges of $F$ are of type II and form a directed cycle. 
	Suppose that these edges are
	$$v_0v_1, v_1v_2, v_2v_0.$$ 
	We label the vertices of $F'$ by $v'_0,v'_1,v'_2$ such that $v'_i=g(v_i)$ for each $i\in\{0,1,2\}$. 
	For $\tau'$, the edges of $F'$ are also of type II. 
	Without loss of generality, we assume that these edges form the directed cycle 
	$$v'_0v'_1, v'_1v'_2, v'_2v'_0$$ 
	(otherwise, for $(\tau')^{-1}$, all faces of $\Gamma'$ are of type II and the edges of $F'$ are directed as above). 
	Let $\Omega(F)$ and $\Omega(F')$ be the sets of all directed edges of $F$ and $F'$, respectively, i.e.
	$$\Omega(F)=\{v_0v_1, v_1v_2, v_2v_0,v_0v_2,v_2v_1,v_1v_0\}$$
	and
	$$\Omega(F')=\{v'_0v'_1, v'_1v'_2, v'_2v'_0,v'_0v'_2,v'_2v'_1,v'_1v'_0\}.$$
	Let also $\mathcal{Z}_F\subset\tau$ and $\mathcal{Z}_{F'}\subset\tau'$ be the sets of all zigzags containing at least one edge from $\Omega(F)$ and $\Omega(F')$ (respectively). 
	For each $i\in\{0,1,2\}$, there exists a unique $j\in\{0,1,2\}$ and a unique sequence $X_i$ of edges of $\Gamma$ such that (with indices taken modulo $3$)
	$$v_iv_{i+1},X_i,v_{j}v_{j+1}$$
	is a subsequence of a zigzag from $\mathcal{Z}_F$ and $X_i$ does not contain any edge from $\Omega(F)$. 
	Similarly, for each $i\in\{0,1,2\}$, there exists a unique $j\in\{0,1,2\}$ and a unique sequence $X'_i$ of edges of $\Gamma'$ such that
	$$v'_iv'_{i+1},X'_i,v'_{j}v'_{j+1}$$
	is a subsequence of a zigzag from $\mathcal{Z}_{F'}$ and $X'_i$ does not contain any edge from $\Omega(F')$. 
	Recall that $v_i=v'_i$ in $\Gamma\#_g\Gamma'$ for every $i\in\{0,1,2\}$. 
	Then, for every such $i$, there exist zigzags 
	$$Z_i=\dots,v_iv_{i+1},X_i,v_{j}v_{j+1},X'_{j},v_kv_{k+1},\dots$$
	and
	$$Z'_i=\dots,v_iv_{i+1},X'_i,v_{l}v_{l+1},X_{l},v_mv_{m+1},\dots$$
	in $\Gamma\#_g\Gamma'$, for some $j,k,l,m\in\{0,1,2\}$. 
	It is possible that $Z_i=Z_j$ or $Z'_i=Z'_j$ or $Z_i=Z'_j$ for $i,j\in\{0,1,2\}$. 
	Since none of these zigzags contains edges of the form $v_{i+1}v_i$, it follows that for the $z$-orientation 
	$$(\tau\setminus\mathcal{Z}_F)\cup(\tau'\setminus\mathcal{Z}_{F'})\cup\{Z_i,Z'_i:i=0,1,2\}$$
	all edges and all faces of $\Gamma\#_g\Gamma'$ are of the same type as in $\Gamma$ and $\Gamma'$. 
\end{proof}

\begin{lemma}\label{lem4}
	Let $F$ and $F'$ be faces in $\Gamma$ and $\Gamma'$, respectively. 
	For every special homeomorphism $g:\partial F\to\partial F'$, the connected sum $\Gamma\#_g\Gamma'$ is $3$-colorable if and only if both $\Gamma$ and $\Gamma'$ are $3$-colorable. 
\end{lemma}
\begin{proof}
	If $\Gamma\#_g\Gamma'$ is $3$-colorable, then its subgraphs $\Gamma$ and $\Gamma'$ are also $3$-colorable. 
	Now, suppose that $\Gamma$ and $\Gamma'$ are $3$-colored using the colors $c_0,c_1,c_2$ and $c'_0,c'_1,c'_2$, respectively. 
	Assume that $v_0,v_1,v_2$ are the vertices of $F$, with each $v_i$ colored in $c_i$ and each vertex $g(v_i)$ of $F'$ colored in $c'_i$. 
	Then, by identifying each color $c_i$ with $c'_i$ and transferring the colorings from $\Gamma$ and $\Gamma'$ to $\Gamma\#_g\Gamma'$, we obtain a $3$-coloring of this connected sum. 
\end{proof}

%n-connected sum
Let $n\geq 2$ and let $\Gamma_i$ be a triangulation of a surface $M_i$ for each $i\in\{1,\dots,n\}$. 
We define recursively the connected sum of $\Gamma_1,\dots,\Gamma_n$ (this sum depends on the corresponding special homeomorphisms). 
By Lemma \ref{lem3}, if each $\Gamma_i$ admits a $z$-orientation with all faces of type II, then this sum also admits a $z$-orientation with all faces of type II. 
Similarly, the sum is $3$-colorable if and only if each $\Gamma_i$ is $3$-colorable.

Let now $\mathcal{T}^n_{k,m}$ be the connected sum of $n$ copies of $\mathcal{T}_{k,m}$ with $k,m\geq 3$ (see Example \ref{ex2}) 
and 
let $\mathcal{P}^n$ be the connected sum of $n$ copies of $\mathcal{P}$ (see Example \ref{ex3}). 
Note that $\mathcal{T}^1_{k,m}=\mathcal{T}_{k,m}$ and $\mathcal{P}^1=\mathcal{P}$. 
Thus, $\mathcal{T}^n_{k,m}$ and $\mathcal{P}^n$ are triangulations of the connected sum of $n$ tori and the connected sum of $n$ real projective planes, respectively. 
Since there exist $z$-orientations of $\mathcal{T}_{k,m}$ (for any $k,m$) and $\mathcal{P}$ such that all their faces are of type II (see Examples \ref{ex2} and \ref{ex3}), also $\mathcal{T}^n_{k,m}$ and $\mathcal{P}^n$ admit $z$-orientations with all faces of type II.
Recall that $\mathcal{P}$ is $3$-colorable and $\mathcal{T}_{k,m}$ is $3$-colorable if and only if both $k,m$ are divisible by $3$.  
Then, also $\mathcal{P}^n$ is $3$-colorable and $\mathcal{T}^n_{k,m}$ is $3$-colorable if and only if both $k,m$ are divisible by $3$. 

If for each face $F$ of $\Gamma$ we add a vertex $v_F$ in the interior of $F$ and connect $v_F$ with every vertex of $F$, then we obtain a new triangulation of $M$ which will be denoted by $T(\Gamma)$. 
By \cite[Proposition 4.1]{T1}, if all faces of $(\Gamma,\tau)$ are of type II, then $T(\Gamma)$ admits a unique $z$-orientation $T(\tau)$ such that all faces of $(T(\Gamma),T(\tau))$ are of type I and $\Gamma$ is the subgraph of $T(\Gamma)$ formed by all edges of type II and their vertices.  
Then, for every surface $M$, the existence of $z$-oriented triangulation with all faces of type II implies the existence of a $z$-oriented triangulation of $M$ containing faces of type I.

%%%%%%%%%%%%%%%%%%%%%%%%%%%%%%%%%%%%%%%%%%%%%%%%%%%%%%%%%%%%%%%%%%%%%%%%%%%%%%%%%%%%%%%%%%%%%%%
\section{Final conclusions}
Let us summarize the above examples. 

(1). If $M$ is the sphere, then Theorem \ref{th1} shows that only one of the following possibilities is realized:
\begin{enumerate}
	\item[$\bullet$] A $z$-oriented triangulation of $M$ has all faces are of type II (for instance, $(\mathcal{O},\tau_2)$ from Example \ref{ex1}) and its Markov chain is not ergodic. 
	\item[$\bullet$] A $z$-oriented triangulation of $M$ contains a face of type I (for example, $(\mathcal{O},\tau_1)$ from Example \ref{ex1}) and its Markov chain is ergodic. 
\end{enumerate}

(2). In the case when $M$ is the connected sum of $n$ tori (in particular, $M$ is the torus for $n=1$), Theorem \ref{th1} guarantees that one of the following three possibilities is realized:
\begin{enumerate}
	\item[$\bullet$] A $z$-oriented triangulation of $M$ is $3$-colorable, all its faces are of type II and the Markov chain is not ergodic 
	(for example, $\mathcal{T}^n_{k,m}$ if the numbers $k,m$ are divisible by $3$). 
	%For example, if the numbers $k,m$ are divisible by $3$, then $\mathcal{T}^n_{k,m}$ is $3$-colorable and admits a $z$-orientation with all faces of type II;
	\item[$\bullet$] A $z$-oriented triangulation of $M$ is not $3$-colorable, all its faces are of type II and the Markov chain is ergodic 
	(for instance, $\mathcal{T}^n_{k,m}$ if at least one of the numbers $k,m$ is not divisible by $3$).  	
	%For instance, if at least one of the numbers $k,m$ is not divisible by $3$, then $\mathcal{T}^n_{k,m}$ is not $3$-colorable and admits a $z$-orientation with all faces of type II; 
	\item[$\bullet$] A $z$-oriented triangulation of $M$ contains a face of type I and the Markov chain is ergodic. 
	For example, $\mathcal{T}^n_{k,m}$ (for any $k,m$) admits a $z$-orientation $\tau_n$ such that all its faces are of type II (as in the previous cases). 
	Then $(T(\mathcal{T}^n_{k,m}),T(\tau_n))$ is a $z$-oriented triangulation of $M$ with all faces of type I. 
\end{enumerate}
	
(3). If $M$ is a connected sum of $n$ real projective planes (in particular, $M$ is the real projective plane for $n=1$), then, by Theorem \ref{th1}, we obtain the following possibilities: 
\begin{enumerate}
	\item[$\bullet$] A $z$-oriented triangulation of $M$ with all faces are of type II is $3$-colorable and the Markov chain is not ergodic. 
	For instance, $\mathcal{P}^n$ is $3$-colorable and admits a $z$-orientation $\tau_n$ with all faces of type II. 
	\item[$\bullet$] A $z$-oriented triangulation of $M$ contains a face of type I and its Markov chain is ergodic. 
	For example, if we take the $z$-oriented triangulation $(\mathcal{P}^n,\tau_n)$ from the previous case, then $(T(\mathcal{P}^n),T(\tau_n))$ is a $z$-oriented triangulation of $M$ with all faces of type I. 
	\item[$\bullet$] A $z$-oriented triangulation of $M$ with all faces are of type II is not $3$-colorable and the Markov chain is ergodic. 
\end{enumerate}
	The third case is impossible for the real projective plane. 
	We do not know it is realized for the remaining non-orientable surfaces. 

%ZMIENIC BIBLIOGRAFIĘ

\end{document}